\documentclass[a4paper,10pt]{amsart}
\usepackage{amsthm,amsmath,amssymb,amsfonts}
\usepackage[T1]{fontenc}
\usepackage[utf8]{inputenc}
\usepackage{verbatim}
\usepackage{marvosym}
\usepackage{stackrel}
\usepackage{graphicx}
\usepackage{eucal}
\usepackage{bbold}
\usepackage[svgnames]{xcolor}
\definecolor{blue(munsell)}{rgb}{0.0, 0.5, 0.69}
\usepackage{lipsum}
\usepackage{svg}
\usepackage{bbm}
\usepackage{caption}
\usepackage{enumitem}
\usepackage{mathtools}
\usepackage{epigraph}
\usepackage{color}
\usepackage{microtype}
\usepackage{relsize}
\usepackage{threeparttable}
\usepackage{amsfonts}
\usepackage{adjustbox}
\usepackage{subfig}
\usepackage{framed}
\usepackage{hyperref}
\hypersetup{
    colorlinks = true,
    linkbordercolor = {red},
   	linkcolor ={teal},
	anchorcolor = {pink},
	citecolor =  {orange},
	filecolor = {blue},
	menucolor = {blue},
	runcolor =  {blue},
	urlcolor = {blue},
}
\usepackage{cleveref}
\usepackage{trimclip}
\allowdisplaybreaks

\DeclareFontFamily{U}{min}{}
\DeclareFontShape{U}{min}{m}{n}{<-> udmj30}{}
\newcommand{\yo}{\!\text{\usefont{U}{min}{m}{n}\symbol{'210}}\!}

\usepackage{stmaryrd}
\usepackage{tikz}
\usepackage{tikz-cd}
\usepackage{quiver}

\theoremstyle{definition}
\newtheorem{thm}{Theorem}[subsection]
\newtheorem*{thm*}{Theorem}
\newtheorem{prop}[thm]{Proposition}
\newtheorem*{prop*}{Proposition}
\newtheorem{lem}[thm]{Lemma}
\newtheorem{cor}[thm]{Corollary}
\newtheorem{defn}[thm]{Definition}
\newtheorem*{defn*}{Definition}

\newtheorem*{war*}{Warning}
\newtheorem{rem}[thm]{Remark}
\newtheorem{constr}[thm]{Construction}
\newtheorem*{constr*}{Construction}
\newtheorem{exa}[thm]{Example}
\newtheorem{conj}[thm]{Conjecture}
\newtheorem{quest}[thm]{Question}
\newtheorem{notat}[thm]{Notation}

\newtheorem{term}[thm]{Terminology}

\newcommand{\textdel}[1]{}

\newcommand{\Hcal}{\mathcal{H}}

\newcommand{\Set}{\mathsf{Set}}
\newcommand{\op}{{}^{\mathrm{op}}}

\newcommand{\mc}[1]{\mathcal{#1}}
\newcommand{\mb}[1]{\mathbf{#1}}
\newcommand{\mbb}[1]{\mathbb{#1}}
\newcommand{\mr}[1]{\mathrm{#1}}

\newcommand{\ms}[1]{\mathsf{#1}}
\newcommand{\lan}{\ms{lan}}
\newcommand{\ran}{\ms{ran}}
\newcommand{\nt}{\Rightarrow}
\newcommand{\LInj}{\ms{LInj}}
\newcommand{\WLInj}{\ms{WLInj}}
\newcommand{\RInj}{\ms{RInj}}
\newcommand{\WRInj}{\ms{WRInj}}
\newcommand{\Inj}{\ms{Inj}}
\newcommand{\Topoi}{\ms{Topoi}}
\newcommand{\LEX}{\ms{LEX}}
\newcommand{\lex}{\ms{lex}}
\newcommand{\cat}{\ms{cat}}
\newcommand{\CAT}{\ms{CAT}}

\newcommand{\psh}{\ms{Psh}}
\newcommand{\sh}{\ms{Sh}}
\newcommand{\surj}{\twoheadrightarrow}
\newcommand{\inj}{\rightarrowtail}
\newcommand{\hook}{\hookrightarrow}
\newcommand{\co}{^{\mr{co}}}

\newcommand{\qsi}[1]{\widetilde{#1}}
\newcommand{\eff}{\Leftrightarrow}
\newcommand{\stt}[1]{\{\,#1\,\}}

\newcommand{\pw}{\ms{pw}}

\makeatletter
\newcommand{\ct@}[2]{%
  \vtop{\m@th\ialign{##\cr
    \hfil$#1\operator@font lim$\hfil\cr
    \noalign{\nointerlineskip\kern1.5\ex@}#2\cr
    \noalign{\nointerlineskip\kern-\ex@}\cr}}%
}
\newcommand{\ct}{%
  \mathop{\mathpalette\ct@{\rightarrowfill@\textstyle}}\nmlimits@
}
\makeatother
\makeatletter
\newcommand{\lt@}[2]{%
  \vtop{\m@th\ialign{##\cr
    \hfil$#1\operator@font lim$\hfil\cr
    \noalign{\nointerlineskip\kern1.5\ex@}#2\cr
    \noalign{\nointerlineskip\kern-\ex@}\cr}}%
}
\newcommand{\lt}{%
  \mathop{\mathpalette\lt@{\leftarrowfill@\textstyle}}\nmlimits@
}
\makeatother

\title{Logic and concepts in the $2$-category of topoi}

\author{Ivan Di Liberti$^{\mathbin{\rotatebox[origin=c]{-45}{$\diamond$}}}$}
\thanks{
$^{{\mathbin{\rotatebox[origin=c]{-45}{$\diamond$}}}}$ Department of Philosophy, Linguistics and Theory of Science, University of Gothenburg, Gothenburg, Sweden. \hfill \href{mailto:diliberti.math@gmail.com}{\sf diliberti.math@gmail.com}
}

\author{Lingyuan Ye$^\diamond$}
\thanks{
$^\diamond$ Department of Computer Science and Technology, University of Cambridge, Cambridge, UK. (Corresponding author). \hfill  \href{mailto:ye.lingyuan.ac@gmail.com}{\sf ye.lingyuan.ac@gmail.com}
}

\begin{document}

\begin{abstract}
We use Kan injectivity to axiomatise concepts in the 2-category of topoi. We showcase the expressivity of this language through many examples, and we establish some aspects of the formal theory of Kan extension in this 2-category (pointwise Kan extensions, fully faithful morphisms, etc.). We use this technology to introduce fragments of geometric logic, and we accommodate essentially algebraic, disjunctive, regular, and coherent logic in our framework, together with some more exotic examples. We show that each fragment $\Hcal$ in our sense identifies a lax-idempotent (relative) pseudomonad $\mathsf{T}^{\Hcal}$ on $\lex$, the $2$-category of finitely complete categories. We show that the algebras for $\mathsf{T}^{\Hcal}$ admit a notion of classifying topos, for which we deliver several Diaconescu-type results. The construction of classifying topoi allows us to define conceptually complete fragments of geometric logic. 
 
\end{abstract}

\maketitle

   {
   \hypersetup{linkcolor=black}
   \tableofcontents
   }

\section*{Introduction}

\subsection*{Landscape}
Geometric logic is a fragment of predicate logic that established itself because of a collection of serendipities:

\begin{itemize}
\item[$\star$] it guarantees great \textit{expressive power}, encompassing many structures of the working mathematician. See, among other sources, \cite[D1.1.7]{elephant2} for a broad collection of examples. 
\item[$\star$] it finds its natural semantics in \textit{Grothendieck topoi} (\cite[D1.2]{elephant2}), which provide a very interesting link to objects of geometric flavour. 
\item[$\star$] it allows for an inherently intuitionistic take on those structures, which is useful in those semantics where classical logic is not available.
\item[$\star$] it features a nice theory of \textit{syntactic categories} \cite[D1.4]{elephant2}. Moreover, its topos-theoretic semantics provides a \textit{completeness theorem} for geometric logic via the theory of classifying topoi \cite{Hakim1972}.
\end{itemize}
 
For these reasons, geometric logic has emerged -- especially in categorical logic -- as the playground/framing to treat and study many aspects of predicate logic. Several relevant fragments of first order logic can be understood as syntactic restrictions over geometric logic. This is particularly clear when reading \cite[D1.1]{elephant2}. We are introduced to a number of \textit{fragments} of geometric logic: \textit{(essentially) algebraic, regular, disjunctive, coherent, and others} (see \cite[Exa. D1.1.7]{elephant2}) and we are provided with a series of results for those fragments.

For each of these fragments, we find a parallel development of results -- each fragment enjoying its own syntactic category and various tailored completeness or representation theorems. Also, they all admit a classifying topos construction corresponding to each type of theory, and in doing so, each identifies a distinctive class of topoi. Coherent theories, for instance, are classified by coherent topoi; regular theories by regular ones; algebraic theories by suitable presheaf topoi. 

While these parallel results are rich and well-developed, there is an elephant in the room\footnote{(No) pun intended.}: we still lack a definition of what a fragment of geometric logic \textit{is}. The notion remains informal, more pattern than principle. As a result, the landscape is modular in spirit but not in practice -- each theorem is reproved for each fragment, following the same strategy, just in a different syntactic costume.

\subsection*{Contribution} This paper develops a framework that allows to:
 \begin{enumerate}
     \item \textit{define} fragments of geometric logic\footnote{For which examples we cover and which we do not, see e.g. \Cref{rem:limitation}.},
     \item accommodate many usual constructions from topos theory (syntactic category, classifying topos for a fragment of logic, conceptual completeness) in a way that is modular with respect to the specification of the fragment,
     \item deliver new theorems and old theorems in a \textit{modular} way. In the introduction of the recent paper \cite{di2024sketches}, the first author has already mentioned several \textit{patterns} of theorems that require a similar framework to be investigated\footnote{See \Cref{abstractnonsense} and \Cref{question}.}.
 \end{enumerate}

We introduce the blueprint of \textit{semantic prescriptions} (\ref{colimits}, \ref{prescribinguniversals}, \ref{rem:semanticcoherent}, \ref{prescriptiononelasttime}) which uses Kan injectivity to prescribe desired/intended properties to the semantics of a special class of geometric theories. For example, we expect categories of models of essentially algebraic theories to have \textit{all limits}, models of regular theories to have \textit{products}, models of coherent theories to have \textit{ultraproducts}. Requiring a topos to have such a property identifies the \textit{fragment of geometric logic} it belongs to.
 
Our notion of logic produces a (Kock--Zöberlein) \textit{doctrine}, consistent with the usual tradition of categorical logic. This should be understood as both a technical device to study logic, but also a sanity check that our notion is not too general. We shall analyse the structure of the paper in details in the next subsection, for the moment, we mention two of our main results. 
 
\begin{thm*}[\ref{constr:fromlotodoc} and \ref{prop:synofcoh}]
Every logic $\Hcal$ has an associated lax-idempotent (relative) pseudomonad $\mathsf{T}^{\mathcal{H}}$ defined over $\lex$. There is a logic $\Hcal_{\text{flat}}$ whose associated $\mathsf{T}^{\Hcal_{\text{flat}}}$ precisely the free pretopos construction over a lex category.
\end{thm*}

\begin{thm*}[\ref{constr:classtopos} and \ref{diaconescu}]
Every (Morita-small) algebra for $\mathsf{T}^\Hcal$ admits a \textit{classifying topos}, producing a 2-functor $\ms{Cl} : \ms{alg}(\ms T^{\Hcal}) \to \Topoi\op$. Moreover, this construction enjoys a Diaconescu-type relative pseudoadjunction as below,
\[\begin{tikzcd}
    {\mathsf{alg}(\mathsf{T}^{\Hcal})} & {\mathsf{Topoi}\op} \\
    {\mathsf{Alg}(\mathsf{T}^{\Hcal})}
    \arrow[""{name=0, anchor=center, inner sep=0}, "{\ms{Cl}}", from=1-1, to=1-2]
    \arrow[from=1-1, to=2-1]
    \arrow[""{name=1, anchor=center, inner sep=0}, from=1-2, to=2-1]
    \arrow["\dashv"{anchor=center, rotate=-90}, draw=none, from=1, to=0]
\end{tikzcd}\]
\end{thm*}


\subsection*{Results and structure of the paper}

In \Cref{sec:framework} we start by recalling the framework of Kan injectivity and Kan injective object from \cite{dlsolo}. Throughout the paper a class of geometric morphism will always be called $\Hcal$, and its associated class of weak right Kan injectives will be called $\WRInj(\Hcal)$. Then we move to show that some toy -- and yet relevant -- classes of topoi can be described as classes of right Kan injectives (up to retract completion).

\begin{thm*}[\ref{chargeom}, \ref{class:essalg}, \ref{class:prop}] 
The following $2$-categories are all categories of weak Kan injectives up to coadjoint retract completion: topoi, free topoi\footnote{And ethereal geometric morphisms between them; see \Cref{defeth}.}, localic topoi.
\end{thm*}

In some cases the class $\Hcal$ can shrink, thanks to the properties of the $2$-category of topoi. This hinges on an analysis of the the theory of Kan extensions in Topoi importing techniques from formal category theory.

We carry this study in \Cref{sec:universal}, where we dive into the general theory of Kan extensions in the $2$-category of topoi. We build on an old observation by Wood (\cite{wood1982abstract}) that $\mathsf{Topoi}$ has a nice proarrow equipment provided by $\mathsf{Topoi}_{\mathsf{lex}}^{\mathsf{co}}$, given by mapping each topos to itself, and each geometric morphism to its direct image,
\[\mathsf{Topoi} \hookrightarrow \mathsf{Topoi}_{\mathsf{lex}}^{\mathsf{co}}.\]
Proarrow equipments allow to define a notion of fully faithful $1$-cell and pointwise Kan extension, so that we can recover typical behaviour of Kan extensions in $\CAT$. The following statement also motivates us to look at the 2-category $\WRInj_{\pw}(\mc H)$ of weak right Kan injectives with pointwise Kan extensions along $\mc H$.

\begin{prop*}[\ref{prop:ranemiffweak}]
    A pointwise Kan extension along a fully faithful morphism is indeed an extension. In this case, weak Kan injecitvity along embeddings reduces to Kan injectivity. 
\end{prop*}

Technical observations of a similar flavour are also investigated for various classes of geometric surjections. In this case the result has a less powerful conceptual meaning, but turns out to be extremely useful in many examples.

\Cref{sec:kanshowoff} showcases a collection of \textit{properties} of a topos that can be presented via Kan extensions. This section serves two purposes, the first one is fairly pedagogical, it holds the reader's hand in the process of building the correct intuition on how to use Kan injectivity. As a secondary purpose, it starts a taxonomy of some relevant classes of the $2$-category of topoi that is of independent interest. We prove that the following classes of topoi can be understood as Kan injectivity classes.

\begin{itemize}
    \item[$\star$] We give a very crisp proof/motivation of the fact that the category of points of a topos (equivalently the category of models of a geometric theory) has directed colimits (\ref{yesdirectedcolimits}).
    \item[$\star$] We show a topos is totally connected iff its category of points has a terminal object, and we characterise this as a Kan injectivity class (\ref{thm:totallyconnected}).
    \item[$\star$] We show that grouplike topoi are Kan injectives, we show they descend along etale surjections (\ref{prop:descentofgrouplike}), and are closed under coproducts (\ref{prop:gpcoprod}). This implies the presheaf topoi on groupoids are grouplike.
    \item[$\star$] We define preorderlike topoi to be those whose generalised points in arbitrary topoi are preorders. Via their Kan injective characterisation, we show that they are closed under localic geometric morphism (\ref{localicpreorder}). This shows that localic topoi are all preorder like. We also discuss the other implication (\ref{conjecturehenry}), contextualizing a recent Mathoverflow question (\cite{373346}).
\end{itemize}


\Cref{sec:logicshowoff} is in some sense a continuation of \Cref{sec:kanshowoff}, but we finally switch to examples of logical flavour, which is the intended focus of this paper. This section is largely inspired by \cite{diliberti2022geometry}, where the first author has observed the original connection between coherent topoi and Kan injectivity. In this section, for many of the usual fragments $\mathsf{L}$ of geometric logic, we exhibit a class of geometric morphisms $\Hcal^{\mathsf{L}}$ such that every topos in that fragment of logic is indeed $\Hcal^{\mathsf{L}}$ weakly injectivity.
 
\begin{table}[!h]\renewcommand{\arraystretch}{1.3}
\begin{tabular}{|c|c|c|}
\hline
\textbf{Logic} --  \textbf{$\mathcal{H}$}  & $\mathsf{WRInj}_\pw(\mathcal{H})$   & $f_*$ preserves \\ \hline

geometric \hfill  $\mathcal{H}_{\text{eth}}$ (\ref{defeth}) & $\mathsf{Topoi}$  \hfill (\ref{chargeom}) &  all colimits \\ \hline

integral \hfill  $\mathcal{H}_{\text{cl}}$ (\ref{exa:complementedfragment}) & $\mathsf{TocTopoi}$ \hfill  (\ref{cor:classtot}) & - \\ \hline

essentially algebraic \hfill  $\mathcal{H}_{\text{all}}$ (\ref{defall}) & $\mathsf{Free}^{*}$ \hfill  (\ref{class:essalg}) & nothing \\ \hline

ess. alg. w. falsum \hfill  $\mathcal{H}_{\text{dom}}$ (\ref{defdom}) & $\mathsf{ClFree}^{*}$  \hfill (\ref{thm:elsealgebraic}) &  initial object \\ \hline

disjunctive \hfill  $\mathcal{H}_{\text{inn}}$ (\ref{definn}) & $\mathsf{DisjTopoi}^{*}$ \hfill  (conj!) & all coproducts \\ \hline

finitary disjunctive \hfill  $\mathcal{H}_{\text{pure}}$ (\ref{defpure}) & $\mathsf{fDisjTopoi}^{*}$ \hfill  (conj!) & coproducts \\ \hline

regular \hfill   $\mathcal{H}_{\text{matte}}$ (\ref{defmatte}) & $\mathsf{RegTopoi}^{*}$ \hfill  (conj!) & epis \\ \hline

coherent \hfill $\mathcal{H}_{\text{flat}}$ (\ref{defflat}) & $\mathsf{CohTopoi}^{*}$ \hfill  (conj!)  & coprod. \& epis \\ \hline

\end{tabular}
    \label{tab:results}
\end{table}

 
In the table\footnote{This table is trying to compress a lot of information in very little space and we recommend the reader to appreciate it as a bird-eye-view over our work. A more technical understanding will require reading the relevant section (see \Cref{sec:logicshowoff} and \Cref{sec:future}). The apex $(-)^{*}$ stands for taking the closure under coadjoint retracts.} above we see on the left the name of the fragment, followed by the corresponding class of morphisms. The middle column indicates the 2-categories of weak injectives. In the last column we give a description of the class $\Hcal^{\mathsf{L}}$. Interestingly a geometric morphism $f$ belongs to $\Hcal^{\mathsf{L}}$ when $f_*$ preserves a relevant shape of colimit\footnote{Besides the fragment $\mc H_{\mr{cl}}$, where it can be chosen as a \emph{single} geometric morphism $\emptyset \hook \Set$; see e.g. \Cref{thm:totallyconnected}.}. For example, for coherent logic we identify the class $\mathcal{H}_{\text{flat}}$ of flat geometric morphisms, and a geometric morphisms $f$ is flat if its direct image preserve finite coproducts and epimorphisms.

This behaviour had already been observed for coherent topoi, which is our motivating example, and hints a connection with the theory of sound doctrines \cite{adamek2002classification}, which this paper does not investigate. Notice that during our analysis two new fragments of logic have emerged: \textit{essentially algebraic logic with falsum} and \textit{integral logic}, and we discuss some of their properties in this paper.


Motivated by the previous examples, \Cref{sec:logic} begins with a general definition of \textit{(fragment of geometric) logic}, which for us starts with the data of a class of geometric morphisms $\Hcal$. A topos \textit{formally belongs} to the logic if it is weakly right injective to all morphisms in $\Hcal$. Then we proceed to show that every topos $\mathcal{E}$ which is formally in a logic admits a \textit{syntactic category} (\ref{syncat}), which should be understood as a tentative axiomatization of the theory classified by the topos in that specific fragment of logic. Such category also admits a canonical site structure (\ref{synsite}). The syntactic category plays a crucial role for us, by precomposing it with the presheaf construction as below,
\[\begin{tikzcd}
	\lex && {\WRInj(\Hcal)^{\text{op}}} & \LEX
	\arrow["{\mathsf{Psh}}"', from=1-1, to=1-3]
	\arrow["{\mathsf{T}^{\Hcal}}"{description}, curve={height=-24pt}, from=1-1, to=1-4]
	\arrow["{\mathsf{Syn}^\Hcal}"', from=1-3, to=1-4]
\end{tikzcd}\]

To each logic $\Hcal$, we can thus associate a lax-idempotent relative pseudomonad $\mathsf{T}^{\Hcal}$ (\ref{thm:laxidempotentmonad}), which connects our notion of logic, to the specification of many \textit{doctrines of first order logic}, as discussed in \cite[introduction and last section]{di2025bi} or in \cite{lex}. The algebras for $\mathsf{T}^{\Hcal}$ are morally categories equipped with property-like natural operations (internal logic) respecting the semantic prescriptions made by ${\Hcal}$. On this topic, our most relevant (and sound for the consistency of the theory) result is the following.

\begin{thm*}[\ref{thm:cohalgebra}]
    $\mathsf{Pretopoi}$ is the 2-category of pseudo-algebras for $\mathsf{T}^{\mathcal{H}_{\text{flat}}}$. 
\end{thm*}

We know that every pretopos has a topos completion that is nothing but the classifying topos of the associated theory. We manage to reproduce this result by proving that every algebra for $\mathsf{T}^{\Hcal}$ indeed admits a classifying topos (\cref{constr:classtopos}). 
\[ \mathsf{Cl}: \mathsf{alg}(\mathsf{T}^{\Hcal}) \to \mathsf{Topoi}\op .\]

For us, a logic is really the \emph{combination} of the semantics prescription $\mc H$, the corresponding class of Kan injectives $\WRInj(\mc H)$ consisting of the theories whose generalised points satisfy the semantic prescription, and the associated lax-idempotent (relative) pseudomonad $\ms T^{\Hcal}$ describing the more concrete syntax. Each of them provides an important perspective on the behaviour of a fragment of logic.

In \Cref{sec:sound}, we start an abstract study of logic, where we establish several Diaconescu-type results relating theories to the classifying topos construction. The first observation is that the classifying topos construction is \emph{conceptually sound}, in the sense that $\ms{Cl}$ in fact factors through $\WRInj_{\pw}(\mc H)$ (\ref{thm:conceptualsound}),
\[ \mathsf{Cl}: \mathsf{alg}(\mathsf{T}^{\Hcal}) \to \WRInj_{\pw}(\mc H)\op .\]
Furthermore, the classifying topos has the right universal property, where we show it is part of a relative pseudoadjunction reifying the classical Diaconescu's theorem (\ref{diaconescu}). Even more interestingly, it offers a new and modular framework to discuss Makkai's conceptual completeness (\ref{relationtoconceptual}) and give a general definition of conceptually complete logic. Finally, \Cref{sec:future} collects many reflections and future directions.

\subsection*{Notations and preliminary definitions}
\begin{defn*}[The $2$-category of topoi]
    Recall the $2$-category of topoi has objects (Grothendieck) topoi, 1-cells geometric morphisms between topoi and $2$-cells natural transformations between left (!) adjoints. We denote this 2-category as $\Topoi$.
\end{defn*}

\begin{defn*}
    We denote $\cat$ ($\lex$) the 2-category of small (!) (left exact) categories, and $\CAT$ ($\LEX$) to denote the 2-category of all (left exact) categories.
\end{defn*}

\subsection*{Acknowledgments}

Both authors are thankful to \textit{Nathanael Arkor} for a very careful analysis of a draft of this paper. We thank \textit{Keisuke Hoshino} for pointing us to \cite[Sec 3]{koudenburg2022formal} in handling pointwise Kan extensions and again to Arkor for confirming our conjecture about Artin gluing by referring to~\cite{wood1985proarrows}. The first author is thankful to \textit{Jérémie Marquès} for pointing out a small inaccuracy in \cite{diliberti2022geometry} which will be later discussed in the paper (\ref{jeremie}). He is also thankful to \textit{Fabian Ruch} for his interest in the project during its initial gestation period. The first named author received funding from Knut and Alice Wallenberg Foundation, grant no. 2020.0199.

\section{The framework of Kan injectivity}\label{sec:framework}

This section recalls the general framework of Kan injectivity and casts it in the $2$-category of topoi. We will also see three examples of Kan injectivity to \textit{warm up}. Kan injectivity was originally studied in the locally posetal case \cite{adamek2015kan} generalized in \cite{dlsolo}. Some \textit{ante litteram} cameos of this technology in the $2$-category of topoi were studied by Johnstone in \cite{johnstone1981injective} and by Bunge and Funk in a series of papers, see for example \cite{bunge1999bicomma}. The most recent application to topos theory is by the first author in \cite{diliberti2022geometry}.

\subsection{Kan injectivity}

Let $\mc K$ be a 2-category. We recall the definition of (weak) Kan injectivity from~\cite{dlsolo}. Given 1-cells $f : A \to B$ and $x : A \to X$, its right Kan extension in $\mc K$ is a 1-cell $\ran_fg$ with a 2-cell $\epsilon : \ran_fx \circ f \nt x$, which is universal for such data, 
\[\begin{tikzcd}
	A & X \\
	B
	\arrow[""{name=0, anchor=center, inner sep=0}, "x", from=1-1, to=1-2]
	\arrow["f"', from=1-1, to=2-1]
	\arrow[""{name=1, anchor=center, inner sep=0}, "{\ran_fx}"', from=2-1, to=1-2]
	\arrow["\epsilon", shorten <=2pt, shorten >=2pt, Rightarrow, from=1, to=0]
\end{tikzcd}\]
In other words, we have an isomorphism natural in $y : B \to X$,
\[ \mc K(A,X)(yf,x) \cong \mc K(B,X)(y,\ran_fx). \]
Dually, we can define the notion of left Kan extension. Furthermore, we say a 1-cell $h : X \to Y$ preserves the right Kan extension, if the pair $(h\circ\ran_fx,h\circ\epsilon)$ also forms a right Kan extension of $hx$ along $f$. For instance, it is straightforward to see that all right adjoints preserves right Kan extensions. Dually, one can define when a 1-cell preserves left Kan extension.

We record some well-known general facts about Kan extensions in a 2-category. Firstly, Kan extensions and adjunctions are mutually definable:

\begin{lem}\label{lem:adjointiffallinj}
    For any 1-cell $f : X \to Y$ in a 2-category $\mc K$, $f$ is a right adjoint iff the right Kan extension $\ran_f1$ exists, and is preserved by $f$, viz. $f\ran_f1 \cong \ran_ff$. In this case all objects are weakly right Kan injective w.r.t. $f$, and these right Kan extensions are absolute in the sense that they are preserved by any 1-cell.
\end{lem}
\begin{proof}
    If $f$ has a left adjoint $l$, then the right Kan extension $\ran_f(-)$ exists for all objects in $\mc K$, which is simply given by pre-composing with $l$. These Kan extensions are thus preserved by any 1-cell. On the other hand, consider the following diagram,
    \[
    \begin{tikzcd}
        X \ar[d, "f"'] \ar[r, equal] & X \ar[d, "f"] \\ 
        Y \ar[ur, dashed, "\ran_f1"description] \ar[r, dashed, "\ran_ff"'] & Y
    \end{tikzcd}
    \]
    By assumption we get a counit $\epsilon : \ran_f1 \circ f \nt 1$. On the other hand, the identity $f \nt f$ induces a unit $\eta : 1 \nt \ran_ff \cong f\circ\ran_f1$. Verifying they satisfy the triangle identities is straight forward.
\end{proof}

 Various closure properties holds for the maps along which an object $X$ has left or right Kan extension:

\begin{lem}\label{lem:verticalcompose}
    For any $X$ in $\mc K$, the class of 1-cells in $\mc K$ along which $X$ has left Kan extension is closed under left adjoints, composition, bicommas and bipushouts. Similarly for right Kan extensions.
\end{lem}
\begin{proof}
    See~\cite[Prop. 1.17]{dlsolo}.
\end{proof}


\begin{defn}[(Weak) Kan injectivity, ~\cite{dlsolo}]
    For any morphism $f : A \to B$ in $\mc K$, we say an object $X$ is \emph{weakly right Kan injective} w.r.t. $f$, if for any $x : A \to X$, the right Kan extension of $x$ along $f$ exists in the 2-category $\mc K$. We say $X$ is \emph{right Kan injective} if furthermore, the counit $\epsilon : \ran_fx \circ f \cong x$ is invertible.
\end{defn}

Dually one can define (weak) left Kan injectives as (weak) right Kan injectives in $\mc K\co$. From the above definition, given any class of maps $\mc H$ in $\mc K$, we can form two locally full sub-2-categories of $\mc K$,
\[ \RInj(\mc H) \subseteq \WRInj(\mc H), \]
where objects are those (weakly) right Kan injective w.r.t. all maps in $\Hcal$, and morphisms are those preserving these right Kan extensions. Similarly, one can define 
\[ \LInj(\mc H) \subseteq \WLInj(\mc H). \]
These sub-2-categories will be called (weak) left/right Kan injectivity classes in $\mc K$. As for the 1-categorical orthogonality classes, the above categories have various closure properties w.r.t. $\mc K$:

\begin{prop}
    For any class of 1-cells $\mc H$ in $\mc K$, $\LInj(\mc H)$, $\WLInj(\mc K)$, $\RInj(\mc H)$ and $\WRInj(\mc H)$ are all closed under bi/pseudolimits in $\mc K$.
\end{prop}
\begin{proof}
    See~\cite[Prop. 1.4]{dlsolo}.
\end{proof}

In a 2-category $\mc K$, we say a map $r : X \to Y$ is a \emph{lali} (resp. \emph{rali}) if it is a left (resp. right) adjoint left inverse, i.e. there is another 1-cell $s : Y \to X$ with $rs \cong 1$ and $r \dashv s$ (resp. $s \dashv r$). We say $Y$ is an \emph{adjoint retract} of $X$, if there is a lali $X \to Y$, i.e. $Y$ arises as the splitting of an idempotent monad structure on $X$. Similarly we say $Y$ is a \emph{coadjoint retract} of $X$, if there is a rali $X \to Y$. 

\begin{lem}
    $\LInj(\mc H)$ and $\WLInj(\mc H)$ are closed under adjoint retracts, and $\RInj(\mc H)$ and $\WRInj(\mc H)$ are closed under coadjoint retracts.
\end{lem}
\begin{proof}
    This is~\cite[Lem. 2.6]{dlsolo}.
\end{proof}

\begin{rem}
    It is also shown in~\cite{dlsolo} that if a 2-category $\mc K$ has bicocomma objects, then any weak left/right Kan injectivity class is also a left/right Kan injectivity class. This is in particular true for $\Topoi$, which indeed has all finite bicolimits. On the other hand, Section~\ref{subsec:embedding} will show for some classes of maps, the notions of weak left/right Kan injectivity and left/right Kan injectivity coincide.
\end{rem}

We consider some simple examples of (weak) right Kan injectivity classes in $\Topoi$, which turns out to classify \textit{fragments} of geometric logic in some way.

\subsection{A trivial example: geometric logic}

\begin{defn}\label{defeth}
    We say a geometric morphism $f : \mc E \to \mc F$ is \emph{ethereal}, if $f_*$ preserves all colimits. Equivalently, $f$ is a right adjoint in $\Topoi$. The class of ethereal maps will be denoted as $\mc H_{\mr{eth}}$.
\end{defn}

\begin{rem}
    Clearly, every local geometric morphism is ethereal, as they correspond to those maps whose left adjoint is in fact a \emph{section}. On the other hand, it is easy to see a \emph{surjective} ethereal map will be \emph{connected}. By~\cite[C3.6.1]{elephant2}, local geometric morphisms are thus exactly surjective ethereal maps in $\Topoi$. Dually, a left adjoint in $\Topoi$ is totally connected iff it is also a surjection.
\end{rem}

\begin{prop}[(Tautological) characterization of geometric logic]\label{chargeom}
   \[ \mathsf{Topoi} = \mathsf{WRInj}(\mathcal{H}_{\mr{eth}}) \] 
\end{prop}
\begin{proof}
    This is a consequence of Lemma~\ref{lem:adjointiffallinj}.
\end{proof}

\begin{rem}
    This proposition introduces a glimpse of the leitmotiv of the paper. We shall characterize special topoi via their injectivity properties with respect to a class of morphisms $\Hcal$. Of course, this is the completely trivial example in which which choose a somewhat negligible class that amounts to not carving out any special topos.
\end{rem}
\subsection{A pet example: essentially algebraic logic}\label{subsec:essalg}

As a non-tautological example, we consider the fragment of essentially algebraic logic. We say a topos $\mc E$ is \emph{free}, if there is a left exact category $\mc C$ that $\mc E \simeq \psh(\mc C)$. The connection between free topoi and essentially algebraic theories is discussed in \cite[D3.1.1]{elephant2}. It is observed in~\cite{johnstone1981injective} that free topoi are weakly right Kan injective w.r.t. \emph{all} geometric morphisms. For reference, we repeat the concrete calculation of such right Kan extension given in~\cite{diliberti2022geometry} here:

\begin{exa}\label{exm:ranfree}
    For any geometric morphism $f : \mc E \to \mc F$ and any geometric morphism $x : \mc E \to \psh(\mc C)$ for a left exact $\mc C$, the right Kan extension $\ran_fx$ exists, and can be computed as follows,
    \[ (\ran_fx)^* \cong \lan_{\yo}(f_*x^*\yo), \quad (\ran_fx)_* \cong \lan_{f_*}x_*, \]
    where $\yo : \mc C \to \psh(\mc C)$ is the Yoneda embedding, and the left Kan extensions on the RHS are computed in $\CAT$, or equivalently in $\LEX$.
\end{exa}

\begin{exa}[Right Kan extension computes direct image]\label{exa:directimagerightKan}
 As a special case, we observe that right Kan extensions compute direct images of geometric morphisms. Let $\Set[\mbb O]$ be the object classifier, where for any topos $\mc E$ we have
    \[ \Topoi(\mc E,\Set[\mbb O]) \simeq \mc E. \]
    Then for any $f : \mc E \to \mc F$, and any $X : \mc E \to \Set[\mbb O]$ viewed as an object in $\mc E$, we have
    \[\begin{tikzcd}
    	{\mc E} & {\Set[\mbb O]} \\
    	{\mc F}
    	\arrow[""{name=0, anchor=center, inner sep=0}, "X", from=1-1, to=1-2]
    	\arrow["f"', from=1-1, to=2-1]
    	\arrow[""{name=1, anchor=center, inner sep=0}, "{\ran_fX \cong f_*X}"', dashed, from=2-1, to=1-2]
    	\arrow["{\epsilon_X}", shorten <=2pt, shorten >=2pt, Rightarrow, from=1, to=0]
    \end{tikzcd}\]
    where the 2-cell is given by the counit $\epsilon_X : f^*f_*X \to X$. This follows from the universal property of $\ran_fX$: For any $F \in \mc F$,
    \begin{align*}
        \mc F(F,\ran_fX) 
        &\cong \Topoi(\mc F,\Set[\mbb O])(F,\ran_fX) \\
        &\cong \Topoi(\mc E,\Set[\mbb O])(Ff,X) \\ 
        &\cong \mc E(f^*F,X).
    \end{align*}
\end{exa}

\begin{defn}\label{defall}
    Let $\mc H_{\mr{all}}$ be the class of all geometric morphisms, and let $\mc H_{\mr{emb}}$ be the class of geometric embeddings.
\end{defn}

\begin{thm}\label{thm:algebraic}
    The following are equivalent:
    \begin{enumerate}
        \item $\mathcal{E}$ is a (coadjoint) retract of a free topos.
        \item $\mathcal{E} \in \RInj(\mc H_{\mr{emb}})$.
        \item $\mc E \in \WRInj_{\pw}(\mc H_{\mr{all}})$.
    \end{enumerate}
    As mentioned, $\WRInj_{\pw}(\mc H)$ consists of topoi in $\WRInj(\mc H)$ having pointwise right Kan extensions along $\mc H$. We will explain the notion of pointwise Kan extension in $\Topoi$ in Section~\ref{sec:universal}.
\end{thm}
\begin{proof}
    (1) $\eff$ (2) is proven in~\cite[C4.3]{elephant2}. (1) $\nt$ (3) is essentially the statement in Example~\ref{exm:ranfree} together with Lemma~\ref{lem:pointwiseretract}. (3) $\nt$ (2) will follow from Proposition~\ref{prop:ranemiffweak}.
\end{proof}

As a passing side note, we observe that the geometry in $\Topoi$ reduces the property of being weakly right Kan injective w.r.t. all geometric morphisms significantly:

\begin{prop}\label{prop:alliffopen}
    $\mc X$ is weakly right Kan injective w.r.t. all geometric morphisms iff $\mc X$ is weakly right Kan injective w.r.t. all open embeddings.
\end{prop}
\begin{proof}
    The only if direction is trivial. For the if direction, recall from~\cite[C3.6]{elephant2} any geometric morphism $f : \mc E \to \mc F$ factors as an open embedding followed by a local geometric morphism,
    \[
    \begin{tikzcd}
        \mc E \ar[rr, "f"] \ar[dr, hook, "i"'] & & \mc F \\
        & \mb{sc}_{\mc F}\mc E \ar[ur, two heads, "q"']
    \end{tikzcd}
    \]
    Since $q$ is in particular ethereal, if $\mc X$ is weakly right Kan injective w.r.t. $i$, then it is also weakly right Kan injective w.r.t. $f$ by Lemma~\ref{lem:verticalcompose}.
\end{proof}

Let $\ms{FreeTopoi}$ be the locally full sub-2-category of free topoi with \emph{ethereal} geometric morphisms between them. The presheaf construction gives us $\lex\op \to \ms{FreeTopoi}$, which in essense takes an essentially algebraic theory to its classifying topos. 

\begin{cor}[Characterisation of essentially algebraic logic]\label{class:essalg}
    \[ \ms{Lex}\op \simeq \ms{FreeTopoi} \subseteq \WRInj(\mathcal{H}_{\mr{all}}), \]
    where the inclusion is a 2-fully-faithful embedding, whose image up to coadjoint retracts are $\mc X \in \WRInj(\mc H_{\mr{all}})$ with pointwise right Kan extensions for $\mc H_{\mr{all}}$.
\end{cor}
\begin{proof}
    By Lemma~\ref{lem:adjointiffallinj}, morphisms in $\ms{WRInj}(\mc H_{\mr{all}})$ are exactly the ethereal geometric morphisms in $\Topoi$. Thus it suffices to show any ethereal geometric morphism $f : \psh(\mc C) \to \psh(\mc D)$ between free topoi is induced by a left exact functor $F : \mc D \to \mc C$. Note that such an $f$ gives an adjoint triple
    \[
    \begin{tikzcd}
    	{\psh(\mc C)} & {\psh(\mc D)}
    	\arrow[""{name=0, anchor=center, inner sep=0}, "{f_*}"{description}, from=1-1, to=1-2]
    	\arrow[""{name=1, anchor=center, inner sep=0}, "{f^*}"', curve={height=18pt}, from=1-2, to=1-1]
    	\arrow[""{name=2, anchor=center, inner sep=0}, "{f^!}", curve={height=-18pt}, from=1-2, to=1-1]
    	\arrow["\dashv"{anchor=center, rotate=-90}, draw=none, from=0, to=2]
    	\arrow["\dashv"{anchor=center, rotate=-90}, draw=none, from=1, to=0]
    \end{tikzcd}
    \]
    $f^!$ above now is the direct image of an essential geometric morphism, thus is induced by a functor $F : \mc D \to \mc C$ since $\mc D,\mc C$ are Cauchy complete (cf.~\cite[A4.1.5]{elephant1}). Now by definition the inverse image $f^*$ is left exact, thus $F$ must be left exact.
\end{proof}

\subsection{A full example: propositional logic} \label{propositionallogic}

It is well-known that localic geometric morphisms are orthogonal to hyperconnected maps in $\Topoi$; see e.g.~\cite[A4.6]{elephant1}. We observe that this orthogonality condition is indeed a Kan injectivity condition:

\begin{thm}\label{thm:localrightKanhy}
    A topos $\mc L$ is localic iff it is right Kan injective w.r.t. hyperconnected geometric morphisms.
\end{thm}
\begin{proof}
    As mentioned the if direction is automatic. On the other hand, let $\mc L$ be localic and $f : \mc E \surj \mc F$ hyperconnected. For any $x : \mc E \to \mc L$, we need to show the unique extension $g$ making the following diagram commute is indeed a right Kan extension in $\Topoi$,
    \[
    \begin{tikzcd}
        \mc E \ar[d, two heads, "f"'] \ar[r, "x"] & \mc L \\
        \mc F \ar[ur, dashed, "g"']
    \end{tikzcd}
    \]
    Since $f^*$ is fully faithful, for any $u : \mc F \to \mc L$ we have
    \[ \Topoi(\mc F,\mc L)(u,g) \cong \Topoi(\mc E,\mc L)(uf,gf) \cong \Topoi(\mc E,\mc L)(uf,x). \]
    This implies that $g$ is indeed the right Kan extension $\ran_fx$.
\end{proof}

In fact, the above proof evidently also applies to the dual case, and shows $g$ is also the \emph{left} Kan extension $\lan_fx$ in $\Topoi$. Indeed, in this case the left and right Kan injectivity \emph{coincide}. Let $\ms{LocTopoi}$ be the full sub-2-category of localic topoi, and define: 

\begin{defn}\label{defhp}
    Let $\mc H_{\mr{hy}}$ the class of hyperconnected geometric morphisms.
\end{defn}

\begin{cor}[Characterisation of geometric propositional logic] \label{class:prop}
    \[ \LInj(\mc H_{\mr{hy}}) = \RInj(\mc H_{\mr{hy}}) = \ms{LocTopoi}. \]
\end{cor}
\begin{proof}
    Proposition~\ref{prop:surjectionfull} and~\ref{prop:corepfullsurjequiv} will imply that $\RInj(\mc H_{\mr{hy}})$ and $\LInj(\mc H_{\mr{hy}})$ coincide as a full sub-2-category of $\Topoi$.
\end{proof}

\section{Universals in the 2-category of topoi}\label{sec:universal}

In this section we will explore more specific results that applies to Kan injectivity in the 2-category $\Topoi$. This will allow us to deliver some of the results we have promised in the previous section, including Theorem~\ref{thm:algebraic} and Corollary~\ref{class:prop}. We will start by connecting Kan extensions with proarrow equipments.

\subsection{A proarrow equipment}

Following~\cite{wood1982abstract}, we use the following notion of proarrow equipment:

\begin{defn}
    A 2-functor $(-)_* : \mc K \to \mc M$ is a \emph{proarrow equipment}, or \emph{equipment} in short, if it is bijective on objects, locally fully faithful, and for each 1-cell $f$ in $\mc K$, $f_*$ has a right adjoint in $\mc M$.
\end{defn}

As already observed in \emph{loc. cit.}, there is an equipment on $\Topoi$ given by 
\[ (-)_* : \Topoi \to \Topoi_{\lex}\co, \]
where the codomain $\Topoi_{\lex}\co$ is a full sub 2-category of $\LEX\co$ on topoi. $(-)_*$ takes each topos $\mc E$ to its underlying category, and takes a geometric morphism to its direct image. Notice that for each geometric morphism $f$, $f_*$ has a left exact left adjoint $f^*$, thus $f^*$ will be a right adjoint of $f_*$ in $\Topoi\co_{\lex}$. And by definition, this 2-functor is locally fully faithful and bijective on objects.

There is a general notion of pointwise Kan extensions for equipments in~\cite{wood1982abstract}:

\begin{defn}
    Let $(-)_* : \mc K \to \mc M$ be an equipment. For any 1-cells $f : A \to B$ and $g : A \to X$ in $\mc K$, if the right Kan extension $(\ran_fg,\epsilon)$ of $g$ along $f$ exists in $\mc K$, then we say it is \emph{pointwise} if $(-)_*$ takes it to a right Kan extension in $\mc M$. 
\end{defn}

\begin{rem}[Different notions of pointwise (coincide)] \label{pointwise}
    There is also a notion of pointwise Kan extensions expressible in any 2-category with comma objects~\cite{10.1007/BFb0063102}. From~\cite[Prop. 40 and 41]{wood1985proarrows}, the proarrow equipment $(-)_* : \Topoi \to \Topoi\co_{\lex}$ has the so-called ``cocartesian tabulations'' in the sense of~\cite{koudenburg2022formal}. Thus, by Prop. 3.22 of \emph{loc. cit.}, the two notions of pointwise Kan extension coincide.
\end{rem}

We observe that for all the examples of (weak) right Kan injectivity classes we have introduced in Section~\ref{sec:framework}, their corresponding right Kan extensions are indeed pointwise w.r.t. the equipment $(-)_* : \Topoi \to \Topoi\co_{\lex}$.

\begin{exa}
    Right Kan extensions into any topos against an ethereal geometric morphism is pointwise, since there they are computed by precomposing with the left adjoint of an ethereal morphism, and 2-functors preserve adjunctions.
\end{exa}

\begin{exa}
    Right Kan extensions into free topoi w.r.t. any geometric morphism is pointwise. This follows from the computation in Example~\ref{exm:ranfree}.
\end{exa}

\begin{exa}
    Right Kan extensions into a localic topos against a hyperconnected geometric morphism will also be pointwise. This follows from a more general fact as shown in Proposition~\ref{prop:injpointwise}.
\end{exa}

\begin{notat}
     For a class $\mc H$ of 1-cells, we will write $\WRInj_{\pw}(\mc H)$ to denote the full sub-2-category of $\WRInj(\mc H)$ with pointwise Kan extensions along $\mc H$. 
\end{notat}

The class $\WRInj_{\pw}(\mc H)$ has some similar closure properties as $\WRInj(\mc H)$. For a start, they are stable under coadjoint retracts:

\begin{lem}\label{lem:pointwiseretract}
    If $X$ has pointwise right Kan extension w.r.t. $f : A \to B$ for the equipment $(-)_* : \mc K \to \mc M$, then so does any coadjoint retract of $X$.
\end{lem}
\begin{proof}
    This follows from the fact that any 2-functor $(-)_*$ preserves adjunctions, and right adjoints preserves right Kan extensions.
\end{proof}

Another typical situation in $\Topoi$ providing pointwise right Kan extension is the case when we have a \emph{presentation}:

\begin{defn}
    Let $\mc H$ be a class of geometric morphisms, and $\mc X \in \WRInj(\mc H)$. We say $\mc X$ has an \emph{$\mc H$-presentation} if there is an embedding $j : \mc X \hook \psh(\mc C)$ in $\WRInj(\mc H)$ into a free topos.
\end{defn}

The upshot is that an object admitting an $\mc H$-presentation will have induced pointwise right Kan extensions against maps in $\mc H$:

\begin{lem}\label{lem:presentationpointwise}
    If $\mc X\in\WRInj(\mc H)$ admits an $\mc H$-presentation, then $\mc X\in\WRInj_\pw(\mc H)$.
\end{lem}
\begin{proof}
    Let $f : \mc E \to \mc F$ be a map in $\mc H$ and let $j : \mc X \hook \psh(\mc C)$ be an $\mc H$-presentation of $\mc X$. Let us look at the following situation,
    \[
    \begin{tikzcd}
        \mc E \ar[d, "f"'] \ar[r, "x"] & \mc X \ar[d, hook, "j"] \\ 
        \mc F \ar[r, dashed, "h"'] \ar[ur, dashed, "g"description] & \psh(\mc C)
    \end{tikzcd}
    \]
    where $g,h$ are the right Kan extension of $x$ and $jx$ along $f$, respectively. By assumption, $jg \cong h$. By Example~\ref{exm:ranfree}, $h$ is computed by
    \[ h_* \cong \lan_{f_*}j_*x_*. \]
    Since $h$ factors through $j$, it then follows that
    \[ j_*g_* \cong h_* \cong j_*j^*h_* \cong j_*\lan_{f_*}x_*. \]
    Now since $j_*$ is fully faithful, it follows that
    \[ g_* \cong \lan_{f_*}x_*, \]
    which implies the right Kan extension $g$ is pointwise.
\end{proof}

The reason we are interested in pointwise Kan extensions is that they interact well with fully faithful maps and surjective maps in a suitable 2-categorical sense. To account for Theorem~\ref{thm:algebraic}, let us first consider embeddings.

\subsection{On the role of embeddings}\label{subsec:embedding}

There is a general notion of fully faithful maps for proarrow equipments given in~\cite{wood1982abstract}:

\begin{defn}
    For a proarrow equipment $(-)_* : \mc K \to \mc M$, a map $f$ in $\mc K$ is fully faithful iff the unit of the adjunction $\eta : 1 \to f^*f_*$ in $\mc M$ is an isomorphism.
\end{defn}

In $\Topoi$, a geometric morphism $f$ is fully faithful iff it is an embedding. The crucial point here is that for any proarrow equipment, the comparison 2-cells for pointwise Kan extensions along fully faithful maps are necessarily invertible. In particular, this implies the following:

\begin{prop}\label{prop:ranemiffweak}
    For a geometric embedding $i : \mc U \hook \mc E$, and for any $x : \mc U \to \mc X$, if the right Kan extension $\ran_ix$ exists in $\Topoi$ and is pointwise, then $\ran_ix \circ i \cong x$.
\end{prop}
\begin{proof}
    By~\cite[Prop. 14]{wood1982abstract}, if $i$ is fully faithful in an equipment, then pointwise weak right Kan extension w.r.t. $i$ will also be strict right Kan extension. Now the claim follows from the fact that an embedding $i$ is fully faithful w.r.t. the equipment $(-)_* : \Topoi \to \Topoi\co_{\lex}$.
\end{proof}

\subsection{On the role of corepresentable surjections}

On the other hand, the Kan injectivity class w.r.t. various notions of surjections will also have certain special properties, which will allow us to prove the promised result on the characterisation of localic topoi in Corollary~\ref{class:essalg}.

\begin{defn}
    We say $f : A \to B$ in $\mc K$ is \emph{corepresentably surjective}, if for all $X$ in $\mc K$ the induced functor is conservative
    \[ - \circ f: \mc K(B,X) \to \mc K(A,X). \]
    $f$ is corepresentably \emph{fully surjective}, if $- \circ f$ is fully faithful.
\end{defn}

In $\Topoi$, given any geometric morphism $f : \mc E \to \mc F$, precomposition with $f$ is equivalently post-composition with the inverse image,
\[
\begin{tikzcd}
    \Topoi(\mc F,\mc X) \ar[r, "-\circ f"] \ar[d, "\simeq"'] & \Topoi(\mc E,\mc X) \ar[d, "\simeq"] \\
    \ms{Geom}(\mc X,\mc F) \ar[r, "f^*\circ-"'] & \ms{Geom}(\mc X,\mc E)
\end{tikzcd}
\]
where $\ms{Geom}(\mc X,\mc F)$ denotes the category of left exact, cocontinuous functors. Given this, we have the following examples:

\begin{exa}\label{exm:surjsurj}
    Surjections are corepresentable surjections in $\Topoi$, since their inverse images are conservative. Connected geometric morphisms are corepresentable full surjections in $\Topoi$, since their inverse images are fully faithful.
\end{exa}

\begin{prop}\label{prop:surjectionfull}
    If $\mc H$ is a family of corepresentable surjections in $\mc K$, then $\LInj(\mc H)$ and $\RInj(\mc H)$ are \emph{full} sub-2-categories of $\mc K$. 
\end{prop}
\begin{proof}
    Let $X,Y \in \RInj(\mc H)$ and let $f : A \to B$ be any map in $\mc H$. Suppose we have $g : X \to Y$. To show it lies in $\RInj(\mc H)$, consider any $x : A \to X$. By assumption,
    \[ g \circ \ran_fx \circ f \cong g \circ x \cong \ran_f(gx) \circ f. \]
    By assumption $-\circ f$ is conservative, thus $g \circ \ran_fx \cong \ran_f(gx)$, which means $g\in\RInj(\mc H)$. Similarly for left Kan extension.
\end{proof}

\begin{prop}\label{prop:corepfullsurjequiv}
    If $\mc H$ is a family of corepresentable full surjections, then $\LInj(\mc H)$ and $\RInj(\mc H)$ coincide as a full sub-2-category. This way, the left/right Kan injectivity condition reduces to an \emph{orthogonality} condition: For any $f : A \to B$ in $\mc H$, $X$ is in the left/right Kan injectivity class iff the functor below is an equivalence,
    \[ - \circ f : \mc K(B,X) \to \mc K(A,X). \]
\end{prop}
\begin{proof}
    Suppose $f$ is corepresentably fully surjective. By definition, $X$ is left Kan injective w.r.t. $f$ iff $-\circ f$ has a left adjoint, whose unit is invertible. If $-\circ f$ is fully faithful, this happens iff it is an equivalence. The same applies to right Kan injectivity, and in this case they coincide. 
\end{proof}

\begin{rem}
    Note that the proofs for the above two proposition crucially relies on strict Kan injectivity. Hence, the results in Proposition~\ref{prop:surjectionfull} and Proposition~\ref{prop:corepfullsurjequiv} does \emph{not} hold for weak Kan injective classes.
\end{rem}

If $\mc H$ entirely consists of corepresentably fully faithful maps, we will write the left/right Kan injectivity class w.r.t. $\mc H$ uniformly as $\ms{Inj}(\mc H)$. In $\Topoi$, they always have pointwise right Kan extensions:

\begin{prop}\label{prop:injpointwise}
    Let $\mc H$ is a family of connected geometric morphisms in $\Topoi$. For any $f : \mc E \to \mc F$ in $\mc H$ and $\mc X\in\Inj(\mc H)$, given $x : \mc E \to \mc X$, the right Kan extension $\ran_fx$ is pointwise, and given by the formula below, 
    \[ (\ran_fx)_* \cong x_*f^*. \]
\end{prop}
\begin{proof}
    For any $F\in\mc F$ and $X\in\mc X$, we have
    \[ \mc X(X,x_*f^*F) \cong \mc E(x^*X,f^*F) \cong \mc E(f^*(\ran_fx)^*X,f^*E) \cong \mc F((\ran_fx)^*X,E). \]
    The second isomorphism uses the fact that $X \in \Inj(\mc H)$, thus $\ran_fx \circ f \cong x$, and in particular $f^*(\ran_fx)^* \cong x^*$; the last isomorphism uses the fact that $f$ is connected, and thus $f^*$ is fully faithful. This shows $(\ran_fx)_* \cong x_*f^*$.

    Now for any $F : \mc F \to \mc X$ in $\LEX$, we have
    \[ \LEX(x_*f^*,F) \cong \LEX(x_*,Ff_*). \]
    The right to left direction is by precomposing with $f^*$, and by the fact that $f_*f^* \cong 1$; the left to right direction is by precomposing with $f_*$, and by the fact that $x_*f^*f_* \cong (\ran_fx)_*f_* \cong x_*$, since $X\in\Inj(\mc H)$. It is easy to verify the two direction are inverses to each other, hence $(\ran_fx)_* = x_*f^*$ is the left Kan extension $\lan_{f_*}x_*$ in $\LEX$, which implies $\ran_fx$ is pointwise.
\end{proof}

Finally, we show that if we have an orthogonality condition, with $\mc H$ a family of corepresentably fully faithful maps in a 2-category $\mc K$, then the Kan injectivity class $\Inj(\mc H)$ satisfies a further closure property. In a 2-category $\mc K$, we say a map $g : Y \to X$ is \emph{orthogonal} to $f : A \to B$, if the following is an equivalence,
\[ \mc K(B,Y) \simeq \mc K(A,Y) \times_{\mc K(A,X)} \mc K(B,X), \]
where the codomain is the pseudopullback of categories. Of course when we take $g$ to be the terminal morphism $X \to 1$, this reduces to the condition of $X$ being orthogonal to $f$.

\begin{lem}\label{lem:orthogalinj}
    Let $\mc H$ be a family of corepresentably full surjections in a 2-category $\mc K$. If $X \in \Inj(\mc H)$ and $g : Y \to X$ is orthogonal to all maps $f : A \to B\in\mc H$, then $Y \in \Inj(\mc H)$.
\end{lem}
\begin{proof}
    When $X\in\Inj(\mc H)$, by definition for any $f : A \to B$ in $\mc H$, $\mc K(B,X) \to \mc K(A,X)$ is an equivalence, thus for the pseudopullback we have
    \[ \mc K(A,Y) \times_{\mc K(A,X)} \mc K(B,X) \simeq \mc K(A,Y). \]
    This way, $\mc K(B,Y) \to \mc K(A,Y)$ is an equivalence for all $f\in\mc H$, thus $Y \in \Inj(\mc H)$.
\end{proof}

In particular, this implies that the Kan injectivity class for a family of connected geometric morphisms in $\Topoi$ is always closed under slicing:

\begin{cor}\label{cor:etalelocal}
    Let $\mc H$ be a family of connected geometric morphisms in $\Topoi$. If $\mc X \in \Inj(\mc H)$ and $\mc Y \to \mc X$ is etale, then $\mc Y \in \Inj(\mc H)$. Similarly, if maps in $\mc H$ are all hyperconnected, then $\mc X\in\Inj(\mc H)$ and $\mc Y \to \mc X$ being localic implies $\mc Y\in\Inj(\mc H)$.
\end{cor}
\begin{proof}
    This follows from the fact that etale maps are orthogonal to connected maps (cf.~\cite[C3.3.4]{elephant2}), and as mentioned hyperconnected maps are orthogonal to localic maps.
\end{proof}

\section{All concepts are Kan extensions} \label{sec:kanshowoff}

The aim of this section is twofold. On one side it offers the reader a collection of examples to familiarise with Kan injectivity and understand how to use and what is its expressive power. On the other side we initiate a descriptive analysis of the $2$-category of topoi and its inhabitants. 

The title of the section is an evident tribute to the very celebrated \cite[X.7]{mac2013categories}, where the theory of Kan extensions is strongly justified because it can encode a vast majority of categorical properties. In \Cref{colimits} (below) we shall see the simplest example of this blueprint in our context. This technology has been generalized in ways that do not have the usual flavour that the category theorist may expect in \cite[Sec. 2]{diliberti2022geometry} to account the construction of ultraproducts, we shall recall in \Cref{prescribinguniversals}.

It follows from our presentation that group-like topoi, totally connected topoi, and other classes with these flavours are closed under pseudolimits, which is a priori everything but evident. 

\begin{rem}[The blueprint of \textit{prescribing universals}] \label{colimits}
The theory of weak left and right Kan injectivity in $\Topoi$ can be used to express the existence of colimits and limits in the category of (generalised) points. Let $\mc A$ be any small category. Recall $\Topoi$ is tensored with $\cat$, i.e. for any topos $\mc E$ there exists a topos $\mc A \otimes \mc E$, where for any $\mc X$ there is a natural equivalence
\[ \Topoi(\mc A\otimes\mc E,\mc X) \simeq \CAT(\mc A,\Topoi(\mc E,\mc X)). \]
In other words, geometric morphisms out of $\mc A \otimes \mc E$ detects $\mc A$-shaped $\mc E$-points of $\mc X$. In fact, the tensor $\mc A \otimes \mc E$ is simply $[\mc A,\mc E]$; see e.g.~\cite[B3.4.7]{elephant1}.

There is a natural map $\mc A \otimes \mc E \to \mc E$ induced by the functor $\mc A \to \mb 1$ into the terminal category. We could then look at the following Kan injectivity problem,
\[
\begin{tikzcd}
    \mc A \otimes \mc E \ar[d] \ar[r] & \mc X \\
    \mc E \ar[ur, dashed]
\end{tikzcd}
\]
The observation is that the Kan injectivity problem w.r.t. $\mc A \otimes \mc E \to \mc E$ detects limits and colimits in $\mc E$-points of $\mc X$, as we shall see in the proposition below.
\end{rem}

\begin{prop}\label{prop:limitsaskan}
    For any category $\mc A$ and topoi $\mc E,\mc X$:
    \begin{itemize}
        \item $\Topoi(\mc E,\mc X)$ has $\mc A$-limits (resp. $\mc A$-colimits) iff $\mc X$ is weakly right (resp. left) Kan injective w.r.t. $\mc A \otimes \mc E \to \mc E$.
    \end{itemize}
    If both $\mc X,\mc Y$ are weakly left (resp. right) Kan injective w.r.t. $\mc A \otimes \mc E \to \mc E$, then for a geometric morphism $f : \mc X \to \mc Y$:
    \begin{itemize}
        \item It preserves left (resp. right) Kan extension along this map, iff the functor 
        \[ f\circ-:\Topoi(\mc E,\mc X) \to \Topoi(\mc E,\mc Y) \] 
        preserves $\mc A$-shaped colimits (resp. limits).
    \end{itemize}
\end{prop}
\begin{proof}
    Recall that $\mc X$ is weakly left Kan injective w.r.t. $\mc A \otimes \mc E \to \mc E$ iff the following restriction functor has a left adjoint,
    \[ \Delta : \Topoi(\mc E,\mc X) \to \Topoi(\mc A \otimes \mc E,\mc X) \simeq \CAT(\mc A,\Topoi(\mc E,\mc X)). \]
    This exactly says $\Topoi(\mc E,\mc X)$ has $\mc A$-shaped colimits. The dual case is completely similar.
\end{proof}


\subsection{Filtered colimits, a tautological example}

Let $\mc A$ be a filtered category. 

\begin{prop}
    For any topos $\mc E$, \emph{all} topoi are weakly left Kan injective w.r.t. $\mc A \otimes \mc E \to \mc E$. In fact, $\WLInj(\stt{\mc A \otimes \mc E \to \mc E}_{\mc E \in \Topoi}) = \Topoi$.
\end{prop}
\begin{proof}
    Recall from~\cite[C3.6.17]{elephant2}, for any filtered category $\mc A$ the geometric morphism
    \[ \mc A \otimes \mc E \simeq [\mc A,\mc E] \to \mc E \]
    is \emph{totally connected}, in particular a left adjoint in $\Topoi$. Hence, every topoi will be weakly left Kan injective w.r.t. this map by Lemma~\ref{lem:adjointiffallinj}.
\end{proof}

\begin{rem}
    In fact, we have an iff condition: For any internal category $\mbb A$ in $\mc E$, the geometric morphism $[\mbb A,\mc E] \to \mc E$ is totally connected iff $\mbb A$ is filtered in $\mc E$ (cf.~\cite[B2.6.8]{elephant1}).
\end{rem}

\begin{cor}  \label{yesdirectedcolimits}
    For any topos $\mc X$, the category $\Topoi(\mc E,\mc X)$ of $\mc E$-points of $\mc X$ has filtered colimits, and any geometric morphism $f : \mc X \to \mc Y$ preserves them.
\end{cor}

\subsection{Terminal object and Totally Connected Topoi}\label{subsec:termandtot}

For the simplest example of a notion of limit, let us consider the case for terminal objects. From our general discussion at the beginning of this section, a topos $\mc X$ where $\Topoi(\mc E,\mc X)$ has a terminal for all $\mc E$ is one that is right Kan injective w.r.t. maps of the following form,
\[
\begin{tikzcd}
    \emptyset \otimes \mc E \simeq \emptyset \ar[d, hook] \ar[r] & \mc X \\
    \mc E \ar[ur, dashed]
\end{tikzcd}
\]

In fact, we have the following result. Let $\ms{pt}(\mc X)$ denote the category of points of a topos $\mc X$, i.e. $\ms{pt}(\mc X) = \Topoi(\Set,\mc X)$.
\begin{thm}\label{thm:totallyconnected}
    The following are equivalent for a topos $\mc X$:
    \begin{enumerate}
        \item For any $\mc E$, $\Topoi(\mc E,\mc X)$ has a terminal object.
        \item $\mc X$ is (weakly) right Kan injective w.r.t. $\emptyset \hook \mc E$ for all $\mc E$.
        \item $\mc X$ is (weakly) right Kan injective w.r.t. complemented embeddings.
        \item $\ms{pt}(\mc X)$ has a terminal object.
        \item $\mc X$ is (weakly) right Kan injective w.r.t. $\emptyset \hook \Set$.
        \item $\mc X$ is totally connected.
    \end{enumerate}
\end{thm}
\begin{proof}
    (1) $\eff$ (2) is automatic. (2) $\eff$ (3) follows from Lemma~\ref{lem:verticalcompose}, since any complemented embedding is a pushout of one of the form $\emptyset \to \mc E$. (1) $\nt$ (4) is trivial, and (4) $\eff$ (5) is again automatic. For (5) $\eff$ (6), recall that $\mc X$ is totally connected iff the unique map $\Gamma_{\mc X} : \mc X \to \Set$ has a right adjoint in $\Topoi$, which is equivalently a terminal object in $\ms{pt}(\mc X)$ since $\Set$ is terminal. Hence it suffices to show (6) $\nt$ (1). Suppose $\mc X$ is totally connected with the point $x : \Set \to \mc X$. By assumption we have $\Gamma_{\mc X} \dashv x$ in $\Topoi$. Now consider an arbitrary topos $\mc E$. $x\Gamma_{\mc E}$ must be the terminal object in $\Topoi(\mc E,\mc X)$: For any map $f : \mc E \to \mc X$,
    \[ \Topoi(\mc E,\mc X)(f,x\Gamma_{\mc E}) \cong \Topoi(\mc E,\Set)(\Gamma_{\mc X}f,\Gamma_{\mc E}) \cong 1, \]
    again due to the fact that $\Set$ is terminal in $\Topoi$.
\end{proof}

For a topological space $X$, $\ms{Sh}(X)$ is totally connected iff it has a dense point, i.e. a point $x$ whose closure is $X$. More generally, for any totally connected topoi we might define the terminal object in $\ms{pt}(\mc X)$ to be the dense point of $\mc X$. Let $\ms{TocTopoi}$ be the locally full sub-2-category of totally connected topoi with geometric morphisms preserving the dense point. Since the right Kan extension of $\emptyset \hook \Set$ to a totally connected topos $\mc X$ simply picks out the dense point, we have:

\begin{cor}[Characterisation of totally connected topoi]\label{cor:classtot}
    \[ \WRInj(\emptyset \hook \Set) = \RInj(\emptyset \hook \Set) = \ms{TocTopoi}. \]
\end{cor}

\begin{rem}
    From~\cite[C3.6.19]{elephant2}, we furthermore know that $\ms{TocTopoi}$ is \emph{reflective} in $\Topoi$, which means that there is an idempotent KZ-pseudomonad on $\Topoi$, whose pseudoalgebras are exactly totally connected topoi (cf.~\cite{dlsolo}). Completely dual to the above development, one can show that
    \[ \WLInj(\emptyset \hook \Set) = \LInj(\emptyset \hook \Set) = \ms{LocTopoi}, \]
    where $\ms{LocTopoi}$ is the locally full sub-2-category of local topoi with geometric morphisms preserving the focal point. This is again reflective in $\Topoi$.
\end{rem}

\subsection{Groupoidal points}\label{subsec:gp}

Let $\mb 2$ and $\mbb I$ be the ``walking morphism'' and ``walking isomorphism'' respectively. Notice that $\mbb I$ is in fact equivalent to the terminal category $\mb 1$, hence we might consider the geometric morphism induced by tensoring with $\mb 2 \to \mb 1$ for a topos $\mc E$,
\[ \mb 2 \otimes \mc E \to \mc E. \]
Such geometric morphisms are especially special, since they are \emph{both totally connected and local}. In particular, they are connected, hence corepresentably fully surjective. Let $\mc H_{\mr{gp}}$ be the class of maps $\stt{\mb 2 \otimes \mc E \to \mc E}_{\mc E\in\Topoi}$.

\begin{thm}
    $\ms{Inj}(\mc H_{\mr{gp}})$ as a full sub-2-category of $\Topoi$ consists of \emph{grouplike} topoi, i.e. those $\mc X$ that $\Topoi(\mc E,\mc X)$ is a groupoid for all $\mc E$.
\end{thm}
\begin{proof}
    By Proposition~\ref{prop:corepfullsurjequiv}, $\mc X \in \ms{Inj}(\mc H_{\mr{gp}})$ iff the following induced functor is an equivalence
    \[ \Topoi(\mc E,\mc X) \hook \Topoi(\mc E,\mc X)^{\to}. \]
    This holds iff $\Topoi(\mc E,\mc X)$ is a groupoid.
\end{proof}

We will write $\ms{GroupTopoi}$ as the full sub-2-category of grouplike topoi. This notion had already appeared in the literature, see cf.~\cite{johnstone1989local} and~\cite[C3.6]{elephant2}. In fact, grouplike topoi can be equivalently characterised as the injectivity classes $\ms{Inj}(\mc H_{\mr{loc}})$ or $\ms{Inj}(\mc H_{\mr{toc}})$, where $\mc H_{\mr{loc}}$ and $\mc H_{\mr{toc}}$ are the class of local and totally connected maps, respectively: 

\begin{cor}[Characterisation of grouplike topoi]\label{cor:chgrouplike}
    \[ \mr{Inj}(\mc H_{\mr{loc}}) = \mr{Inj}(\mc H_{\mr{toc}}) = \mr{Inj}(\mc H_{\mr{gp}}) = \ms{GroupTopoi}. \]
\end{cor}
\begin{proof}
    One direction simply follows from the fact $\mc H_{\mr{loc}} \supseteq \mc H_{\mr{gp}} \subseteq \mc H_{\mr{toc}}$. For the other direction, it follows from~\cite[C3.6]{elephant1} that if $\mc X$ is grouplike, then 
    \[ - \circ f : \mc K(\mc F,\mc X) \simeq \mc K(\mc E,\mc X) \]
    is an equivalence for any local or totally connected $f : \mc E \to \mc F$.
\end{proof}

We observe that the above characterisation allows us to establish some stability properties of grouplike topoi. As an easy observation, recall from Corollary~\ref{cor:etalelocal} we know that grouplike topoi will be closed under slicing. However, the more interesting direction is the following statement, which says being grouplike also ``descends along'' etale surjections:

\begin{prop}\label{prop:descentofgrouplike}
    Let $\mc Y \surj \mc X$ be an etale surjection. If $\mc Y$ is grouplike, so is $\mc X$.
\end{prop}
\begin{proof}
    We use the characterisation in Corollary~\ref{cor:chgrouplike}. Notice that for any totally connected map $r : \mc E \surj \mc F$ with centre $c : \mc F \to \mc E$, $\mc X$ is left Kan injective w.r.t. $r$ iff for all $x : \mc E \to \mc X$, $xcr \cong x$. Suppose $X : \mc Y = \mc X/X \surj \mc X$ is an etale surjection with $\mc Y$ grouplike. We form the following pullbacks of $X$ along $x$ and $c$,
    \[
    \begin{tikzcd}
    	{\mc F/c^*x^*X} & {\mc E/x^*X} & {\mc X/X} \\
    	{\mc F} & {\mc E} & {\mc X}
        \arrow["\lrcorner"{anchor=center, pos=0.125}, draw=none, from=1-1, to=2-2]
        \arrow["\lrcorner"{anchor=center, pos=0.125}, draw=none, from=1-2, to=2-3]
    	\arrow[""{name=0, anchor=center, inner sep=0}, "{c/x^*X}"', from=1-1, to=1-2]
    	\arrow["{c^*x^*X}"', two heads, from=1-1, to=2-1]
    	\arrow[""{name=1, anchor=center, inner sep=0}, "{r/x^*X}"', curve={height=12pt}, two heads, from=1-2, to=1-1]
    	\arrow["{x/X}", from=1-2, to=1-3]
    	\arrow["{x^*X}"{description}, two heads, from=1-2, to=2-2]
    	\arrow["X", two heads, from=1-3, to=2-3]
    	\arrow[""{name=2, anchor=center, inner sep=0}, "c"', from=2-1, to=2-2]
    	\arrow[""{name=3, anchor=center, inner sep=0}, "r"', curve={height=12pt}, two heads, from=2-2, to=2-1]
    	\arrow["x"', from=2-2, to=2-3]
    	\arrow["\dashv"{anchor=center, rotate=-90}, draw=none, from=1, to=0]
    	\arrow["\dashv"{anchor=center, rotate=-90}, draw=none, from=3, to=2]
    \end{tikzcd}
    \]
    In particular, we again get a totally connected map
    \[ r/x^*X : \mc E/x^*X \surj \mc F/c^*x^*X, \]
    where its inverse image takes any $A \to c^*x^*X$ to the transpose $r^*A \to x^*X$. Since $\mc X/X$ is grouplike, it follows that
    \[ x/X \cong (x/X) \circ (c/x^*X) \circ (r/x^*X), \]
    which translates to the fact that
    \[ x^*X \cong r^*c^*x^*X. \]
    This implies the left square also commutes with $r$, i.e. we have
    \[ (c^*x^*X) \circ (r/x^*X) \cong r \circ (x^*X), \]
    since for any $F\in\mc F$, the inverse image of the LHS takes it to the transpose of $F \times c^*x^*X \to c^*x^*X$, which is
    \[ r^*F \times r^*c^*x^*X \cong r^*F \times x^*X \to x^*X. \]
    which is the same as $(x^*X)^*r^*F$. Now we can observe that
    \begin{align*}
        xcr \circ (x^*X)
        &\cong xc \circ (c^*x^*X) \circ (r/x^*X) \\
        &\cong x \circ (x^*X) \circ (c/x^*X) \circ (r/x^*X) \\
        &\cong X \circ (x/X) \circ (c/x^*X) \circ (r/x^*X) \\
        &\cong X \circ x/X \\
        &\cong x \circ (x^*X)
    \end{align*}
    The third isomorphism holds since $\mc X/X$ is grouplike, and all the others follows from the commutativity of the above diagram. Finally, since $x^*X$ is a surjection, it follows that $xcr \cong x$,
    which implies $\mc X$ is grouplike.
\end{proof}

\begin{cor}
    For any group $G$, $[G,\Set]$ is grouplike.
\end{cor}
\begin{proof}
    This follows from the observation that the surjective point $\Set \surj [G,\Set]$ is in fact \emph{etale}, since this point is induced by $\Set \simeq [G,\Set]/G$. 
\end{proof}

\begin{rem}
    Notice that for a monoid $M$ there is also a surjective point $\Set \surj [M,\Set]$, but this is \emph{not} etale. In fact, it is etale iff $M$ is a group.
\end{rem}

We may also observe that grouplike topoi are closed under coproducts:

\begin{prop}\label{prop:gpcoprod}
    If we have a family of grouplike topoi $\stt{\mc X_i}_{i\in I}$, then the coproduct $\coprod_{i\in I}\mc X_i$ is also grouplike.
\end{prop}
\begin{proof}
    Suppose we have a totally connected map $r : \mc E \to \mc F$ with centre $c$. Consider a map $x : \mc E \to \coprod_{i\in\mc I}\mc X_i$. Since the inclusion $\mc X_i \inj \coprod_{i\in I}\mc X_i$ is again etale, the first half of the proof of Proposition~\ref{prop:descentofgrouplike} still works, i.e. we now have a commutative diagram
    \[
    \begin{tikzcd}
    	{\mc F_i} & {\mc E_i} & {\mc X_i} \\
    	{\mc F} & {\mc E} & {\coprod_i\mc X_i}
    	\arrow[""{name=0, anchor=center, inner sep=0}, "{c_i}"', from=1-1, to=1-2]
    	\arrow[two heads, from=1-1, to=2-1]
    	\arrow[""{name=1, anchor=center, inner sep=0}, "{r_i}"', curve={height=12pt}, two heads, from=1-2, to=1-1]
    	\arrow["{x_i}", from=1-2, to=1-3]
    	\arrow[two heads, from=1-2, to=2-2]
    	\arrow[two heads, from=1-3, to=2-3]
    	\arrow[""{name=2, anchor=center, inner sep=0}, "c"', from=2-1, to=2-2]
    	\arrow[""{name=3, anchor=center, inner sep=0}, "r"', curve={height=12pt}, two heads, from=2-2, to=2-1]
    	\arrow["x"', from=2-2, to=2-3]
    	\arrow["\dashv"{anchor=center, rotate=-91}, draw=none, from=1, to=0]
    	\arrow["\dashv"{anchor=center, rotate=-90}, draw=none, from=3, to=2]
    \end{tikzcd}
    \]
    In other words, the totally connected geometric morphism $r$ has also been decomposed into a coproduct family of totally connected geometric morphisms. Since each $x_i$ is grouplike, we have $x_i \cong x_ic_ir_i$, which implies 
    \[ x = \coprod_ix_i \cong \coprod x_ic_ir_i \cong xcr, \]
    hence $\coprod_i\mc X_i$ is also grouplike.
\end{proof}

\begin{cor}
    For any groupoid $\mc G$, $[\mc G,\Set]$ is grouplike.
\end{cor}
\begin{proof}
    $[\mc G,\Set]$ is a coproduct in $\Topoi$ of group actions $[G_i,\Set]$, for each connected component $G_i$ of $\mc G$.
\end{proof}

\subsection{Preorder points}

Similar to the above case, consider the category $\mbb P$ of the ``walking parallel morphisms'', with the quotient $q : \mbb P \surj \mb 2$ identifying the parallel pair. For any $\mc E$, it induces a geometric morphism
\[ q \otimes \mc E : \mbb P \otimes \mc E \to \mb 2 \otimes \mc E, \]
which in this case is \emph{hyperconnected}, hence also corepresentably fully surjective. Let $\mc H_{\mr{po}}$ be the class of maps $q \otimes \mc E$ for all $\mc E$.

\begin{thm}
    $\ms{Inj}(\mc H_{\mr{po}})$ as a full sub-2-categories of $\Topoi$ consists of \emph{preorderlike} topoi, i.e. those $\mc X$ that $\Topoi(\mc E,\mc X)$ is a preorder for all $\mc E$.
\end{thm}
\begin{proof}
    A topos $\mc X$ is Kan injective w.r.t. $q \otimes \mc E$ iff we have an equivalence
    \[ \Topoi(\mc E,\mc X)^\to \simeq \Topoi(\mc E,\mc X)^{\rightrightarrows}, \]
    which happens iff $\Topoi(\mc E,\mc X)$ is a preorder.
\end{proof}

\begin{cor} \label{localicpreorder}
    If $\mc X$ is preorderlike, and $\mc Y \to \mc X$ is localic, then so is $\mc Y$.
\end{cor}
\begin{proof}
    This follows from Corollary~\ref{cor:etalelocal}.
\end{proof}

\begin{rem}
    It is worth mentioning that the local properties of the class of grouplike topoi developed in Section~\ref{subsec:gp} is \emph{not} applicable here. Although for any $\mc E$ the geometric morphism $q \otimes \mc E : \mbb P \otimes \mc E \to \mb 2 \otimes \mc E$ has a section $i \otimes \mc E$, induced by a section $i : \mb 2 \to \mbb P$ of $q$, $q$ and $i$ are \emph{not} adjoints to each other. Hence, one cannot in particular infer $\mc X$ belongs to $\Inj(\mc H)$, even if for all $x : \mbb P \otimes \mc E \to \mc X$, $xiq \cong x$.
\end{rem}

However, one can still try to establish a descent property for preorderlike topoi as well. For instance, from the answer to the MathOverflow~\cite{373346} question, the following seems to be true:

\begin{conj} \label{conjecturehenry}
    Let $\mc X$ be a topos which is preorderlike, and let $\mc G$ be a group object in $\Topoi$. If there is a \emph{free} group action $\mu : \mc X \times \mc G \to \mc X$, then the descent object $\mc X/\mc G$ of $\mc G$-equivariant sheaves on $\mc X$ is also preorderlike. 
\end{conj}

\section{Logic as concept} \label{sec:logicshowoff}

In Section~\ref{sec:framework} we have seen that geometric logic, essentially algebraic logic, and propositional logic can be described as (weak) right Kan injectivity classes in $\Topoi$. In this section we will furthermore describe the following fragments in essentially the same way:

\begin{itemize}
    \item[(\ref{else})] essentially algebraic logic with falsum,
    \item[(\ref{disj})] (finitary) disjunctive logic,
    \item[(\ref{regular})] regular logic,
    \item[(\ref{coherent})] coherent logic.
\end{itemize}

This section is largely preparatory, and the general definition of logic and its study will come in the next section, taking inspiration from the taxonomy that we run here.
The proof that these fragments can be described by weak right Kan injectivity follows a generic pattern which will be used repeatedly. The argument is essentially contained in~\cite{diliberti2022geometry}, and we record it here:

\begin{prop}\label{prop:genericproof}
    Let $\mc H$ be a class of geometric morphisms. Let $\mc X \simeq \sh(\mc C,J)$ be a topos with $\mc C$ left exact. If for any $x : \mc E \to \mc X$ the induced functor $f_*x^*y : \mc C \to \mc F$ is $J$-continuous where $y = j^*\yo$, then $\mc X \in \WRInj(\mc H)$ and $j : \mc X\hook\psh(\mc C)$ is an $\mc H$-presentation of $\mc X$.
\end{prop}
\begin{proof}
    Let $f : \mc E \to \mc F$ be a map in $\mc H$. Following~\cite{diliberti2022geometry}, consider the following situation for any $x : \mc E \to \mc X$,
    \[
    \begin{tikzcd}
        \mc E \ar[d, "f"'] \ar[r, "x"] & \mc X \ar[d, hook, "j"] \\ 
        \mc F \ar[r, "g = \ran_{f}xj"'] \ar[ur, dashed] & \psh(\mc C)
    \end{tikzcd}
    \]
    It suffices to show the right Kan extension $g$ factors through $j$. By Diaconescu's, $g$ is induced by the following left exact functor
    \[ g^*\yo \cong \lan_{\yo}(f_*x^*j^*\yo) \circ \yo \cong f_*x^*y, \]
    where we have used the computation in Example~\ref{exm:ranfree}. By assumption, this functor is $J$-continuous, thus $g$ indeed factors through $j$.
\end{proof}

In particular, such topoi $\mc X$ belongs to $\WRInj_{\pw}(\mc H)$ by Lemma~\ref{lem:presentationpointwise}.

\subsection{Essentially algebraic logic with falsum} \label{else}

Recall from~\cite{SGA4} the notion of dominant geometric morphisms: 

\begin{defn}\label{defdom}
    A geometric morphism $f$ is \emph{dominant} if $f_*$ preserves the initial object. We use $\mc H_{\mr{dom}}$ to denote the class of dominant geometric morphisms.
\end{defn}

Immediately from the above definition, we can conclude that the category of generalised points of $\mc X \in \WRInj(\mc H_{\mr{dom}})$ has non-empty limits:

\begin{lem}
    If $\mc X \in \WRInj(\mc H_{\mr{dom}})$, then $\Topoi(\mc E,\mc X)$ has non-empty limits for all $\mc E$.
\end{lem}
\begin{proof}
    For any non-empty category $\mc C$, the geometric morphism
    \[ \mc C \otimes \mc E \to \mc E \]
    is dominant, since its direct image is computed by taking the limit along $\mc C$, and non-empty limits commutes with empty colimit in $\mc E$.
\end{proof}

Just like the case for essentially algebraic logic described in Section~\ref{subsec:essalg}, this weak right Kan injectivity class corresponds to a fragment of logic:

\begin{defn}
    We call a geometric theory $\mbb T$ an \emph{essentially algebraic theory with falsum}, if the signature of $\mbb T$ is algebraic, and the axioms of $\mbb T$ only involves $\top,\wedge,=,\bot$ and provably unique existential quantifiers.
\end{defn}

\begin{rem}
    Every geometric formula $\Phi$ is equivalent to one of the form $\bigvee_{i\in I}\exists \vec x\phi_i$, where each $\phi_i$ is a finite conjunction of atomic formulas. Hence, any geometric sequent of the form $\Phi \vdash \bot$ is equivalent to a family of sequents $\stt{\phi_i \vdash \bot}_{i\in I}$. This way, an essentially algebraic theory with falsum is exactly an essentially algebraic theory where we add a single geometric sequent of the form $\Phi \vdash \bot$. 
\end{rem}

\begin{exa}
    The theory of characteristic 0 rings is an example. It is the theory of rings plus the axiom $\bigvee_{1\le n}n = 0 \vdash \bot$.
\end{exa}

From the definition, one can characterise the classifying topoi for essentially algebraic theories with falsum. Let $\mc C$ be any category. We say a full subcategory $\mc D \subseteq \mc C$ is a \emph{cosieve} in $\mc C$, if for any $d\in\mc D$, if there is a morphism $d \to c$, then $c \in \mc D$. Recall from~\cite[A4.5.5]{elephant1}, cosieves on $\mc C$ corresponds exactly to \emph{closed subtopoi} of $\psh(\mc C)$, where we may identify $\psh(\mc D) \simeq \sh(\mc C,J)$ with $J$ assigning an empty covering sieve for any $c\in\mc C$ with $c\not\in\mc D$.

\begin{lem}\label{lem:classifyingofelsealgebraic}
    A topos is the classifying topos for an essentially algebraic theory with falsum iff it is of the form $\psh(\mc D)$ for $\mc D$ a cosieve in a left exact category $\mc C$, i.e. iff it is a closed subtopos of a free one.
\end{lem}
\begin{proof}
    From the syntactic description of an essentially algebraic theory with falsum, its classifying topos will be of the form $\sh(\mc C,J)$, where $J$ only assigns empty covers. Let $\mc D$ be the full subcategory of $\mc C$ where $c\in\mc D$ iff the empty sieve is not a $J$-cover of $c$. Since $J$-covers are closed under pullback, $\mc D$ is a cosieve. By the comparison lemma, we have $\psh(\mc D) \simeq \sh(\mc C,J)$.
\end{proof}

\begin{rem}
    Alternatively, from a logical perspective, closed subtopoi of $\Set[\mbb T]$ for the classifying topos of a geometric theory $\mbb T$ corresponds exactly to quotients of $\mbb T$ by adding a sequent of the form $\Phi \vdash \bot$; see e.g.~\cite[Sec. 4.2.2]{caramello2018theories}.
\end{rem}

\begin{prop}\label{prop:cosievegenericallyalgebraic}
    For any essentially algebraic theory with falsum $\mbb T$, $\Set[\mbb T] \in \WRInj(\mc H_{\mr{dom}})$, and any such topos admits an $\mc H_{\mr{dom}}$-presentation.
\end{prop}
\begin{proof}
    By Lemma~\ref{lem:classifyingofelsealgebraic} it suffices to show $\psh(\mc D) \simeq \sh(\mc C,J)$ lies in $\WRInj(\mc H_{\mr{dom}})$ for any cosieve $\mc D$ in a left exact category $\mc C$. Let $j : \psh(\mc D) \hook \psh(\mc C)$ be the embedding. By Proposition~\ref{prop:genericproof}, it suffices to show for any dominant map $f : \mc E \to \mc F$ and $x : \mc E \to \psh(\mc D)$, the induced map $f_*x^*y$ is $J$-continuous with $y = j^*\yo$. Now suppose $c\in\mc C$ has an empty $J$-covering, then $y c \cong \emptyset$. Since $f$ is dominant, $f_*x^*y c \cong \emptyset$, hence it is $J$-continuous.
\end{proof}

In this case, we can indeed completely describe the relationship between the classifying topoi of essentially algebraic theories with falsum, and the formal objects in $\WRInj(\mc H)$. As a similar result to Theorem~\ref{thm:algebraic}, we have:

\begin{thm}[Characterisation of essentially algebraic logic with falsum]\label{thm:elsealgebraic}
    TFAE:
    \begin{enumerate}
        \item $\mc X \in \WRInj_{\pw}(\mc H_{\mr{dom}})$.
        \item $\mc X$ is a coadjoint retract of $\Set[\mbb T]$ for an essentially algebraic theory with falsum $\mbb T$.
    \end{enumerate}
\end{thm}
\begin{proof}
    (2) $\nt$ (1) follows from Proposition~\ref{prop:cosievegenericallyalgebraic}. For (1) $\nt$ (2), suppose $\mc X \in \WRInj(\mc H)$, and let $k : \mc X \hook \psh(\mc C)$ be an inclusion into a free topos. Define a cosieve $\mc D$ of $\mc C$ where $c\in\mc D$ iff $k^*c \not\cong \emptyset$. Now we claim that $k$ factors through the inclusion $j : \psh(\mc D) \hook \psh(\mc C)$ as
    \[ j \circ i : \mc X \hook \psh(\mc D) \hook \psh(\mc C). \]
    Since both $k$ and $j$ are embeddings, it suffices to show that for any $X\in\mc X$, $k_*X$ is a $j$-sheaf, i.e. $j_*j^*k_*X \cong k_*X$. Let us use $i_*$ to denote the functor $j^*k_*$. For any $X\in\mc X$, since $j_*$ is given by right Kan extension, we have for any $c\in\mc C$,
    \[ j_*i_*X(c) \cong \lt_{d:\mc D/c}i_*X(d). \]
    There are now two cases: (1) If $c\in\mc D$, then $j_*i_*X(c) \cong i_*X(c) \cong k_*X(c)$; (2) If $c\not\in\mc D$, which by definition $k^*c \cong \emptyset$, then $j_*i_*X(c) \cong 1$. But in this case we must also have $k_*X(c) \cong 1$, since we have
    \[ k_*X(c) \cong \psh(\mc C)(c,k_*X) \cong \mc X(k^*c,X) \cong 1. \]
    Thus indeed $k$ factors through $j$ along $i : \mc X \hook \psh(\mc D)$. 
    
    By Lemma~\ref{lem:classifyingofelsealgebraic}, $\psh(\mc D)$ is the classifying topos $\Set[\mbb T]$ for an essentially algebraic theory with falsum. Since $\mc X$ has pointwise right Kan extension against dominant maps, by Proposition~\ref{prop:ranemiffweak} it thus suffices to show $i$ is dominant. This is indeed the case, since for any $d\in\mc D$,
    \[ i_*\emptyset(d) \cong k_*\emptyset(d) \cong \psh(\mc C)(d,k_*\emptyset) \cong \mc X(k^*d,\emptyset) \cong \emptyset, \]
    because by definition $k^*d\not\cong\emptyset$ in $\mc X$, and the initial object in any topos is strict.
\end{proof}

\subsection{Disjunctive logic} \label{disj}

Below we start by recalling the notion of pure geometric morphisms, and introduce the related notion of innocent geometric morphism.

\begin{defn}\label{defpure}
    A geometric morphism $f$ is \emph{pure} if $f_*$ preserves finite coproducts. We write $\mc H_{\mr{pure}}$ for the class of pure geometric morphisms.
\end{defn}

\begin{defn}\label{definn}
    A geometric morphism is \emph{innocent} if $f_*$ preserves all coproducts. We write $\mc H_{\mr{inn}}$ for the class of innocent geometric morphisms. 
\end{defn}

\begin{rem}
    Pure geometric morphisms were introduced in~\cite{johnstone1981factorization}, and are usually studied in the context of closed maps between topoi, cf.~\cite[III.5]{moerdijk2000proper} and~\cite[C3.4]{elephant2}. There seems to be no study of innocent maps, except we know connected geometric morphisms are innocent (cf.~\cite[C3.4.14]{elephant2}).
\end{rem}

\begin{defn}
    A topos $\mathcal{X}$ is \emph{formally disjunctive} (resp. \emph{formally finitary disjunctive}), if it is weakly right Kan injective w.r.t. \emph{innocent} (resp. \emph{pure}) geometric morphisms.
\end{defn}

Again, we can immediately conclude that the generalised category of points of a formally disjunctive topos has connected limits:

\begin{lem} \label{connected limits}
    If $\mc X$ is formally disjunctive, then $\Topoi(\mc E,\mc X)$ has connected limits for all $\mc E$.
\end{lem}
\begin{proof}
    For any connected category $\mc C$, the geometric morphism will be innocent, $ \mc C \otimes \mc E \to \mc E$ since this geometric morphism is in particular connected.
\end{proof}

Following~\cite{johnstone1979syntactic}, we call a geometric theory $\mbb T$ \emph{disjunctive} (resp. \emph{finitary disjunctive}), if the axioms in $\mbb T$ only involves provably unique existential quantifier and provably disjoint disjunctions (resp. provably disjoint finite disjunctions). We say a topos is (finitary) disjunctive if it is the classifying topos of a (finitary) disjunctive theory.

\begin{exa}
    For any category $\mc C$ with finite multilimits, the theory of flat $\mc C$-functors is disjunctive (cf.~\cite{johnstone1979syntactic}). The coherent theory of fields is an important example of a finitary disjunctive theory.
\end{exa}

To illustrate the idea that the classifying topoi of a (finitary) disjunctive theory will be (finitary) disjunctive, we first consider the classifying topos of flat $\mc C$-functors for a multicomplete category $\mc C$:

\begin{prop}\label{preshdisj}
    If $\mc C$ has finite multilimits, then $\psh(\mc C)$ is formally disjunctive, with pointwise right Kan extensions against innocent geometric morphisms.
\end{prop}
\begin{proof}
    Let $f : \mc E \to \mc F$ be an innocent geometric morphism, and consider any $x : \mc E \to \psh(\mc C)$. To show the following right Kan extension exists,
    \[
    \begin{tikzcd}
        \mc E \ar[d, "f"'] \ar[r, "x"] & \psh(\mc C) \\
        \mc F \ar[ur, dashed, "h"']
    \end{tikzcd}
    \]
    following the computation for free topoi given in Example~\ref{exm:ranfree}, it suffices to show the following left Kan extension in $\CAT$
    \[ h^* = \lan_{\yo} f_*x^*\yo \]
    is left exact, i.e. $f_*x^*\yo$ is flat. Since $\mc C$ has finite multilimits, by \cite[Prop. 1.1]{hu1995limits} it is flat iff it merges finite multilimits. Suppose we have a finite diagram $\stt{A_i}$ in $\mc C$. By~\cite[Cor. 2.5]{hu1995limits}, the finite limit $\lt{\yo_{A_i}}$ is a coproduct of representables, which we write as $\coprod_{s}\yo_{B_s}$. This way, we have
    \[ \coprod_{s} f_*x^*\yo_{B_s} \cong f_*x^*\coprod\yo_{B_s} \cong f_*x^*\lt\yo_{A_i} \cong \lt f_*x^*\yo_{A_i}. \]
    The first isomorphism uses the fact that $f$ is innocent. Hence, the composite $f_*x^*\yo$ merges finite multilimits, which implies $h^*$ as the above left Kan extension is left exact, thus provides a geometric morphism $\mc F \to \psh(\mc C)$.
\end{proof}

For a left exact category $\mc C$, we say a Grothendieck topology $J$ on $\mc C$ is \emph{disjunctive} (resp. \emph{finitary disjunctive}), if $J$ is generated by covers (resp. finite covers) of the form $\stt{X_i \inj X}_{i\in I}$, where $X_i \inj X$ are monomorphisms and the empty sieve is a $J$-cover for $X_i \wedge X_j$ whenever $i \neq j$. A site $(\mc C,J)$ of such form is called a disjunctive (resp. finitary disjunctive) site.

\begin{lem}
    A topos is (finitary) disjunctive iff it is of the form $\sh(\mc C,J)$ for a (finitary) disjunctive site.
\end{lem}
\begin{proof}
    This follows from the syntactic description of disjunctive sequents given in~\cite[Cor. 3.5]{johnstone1979syntactic}.
\end{proof}

\begin{exa}
    Recall that a \emph{Gaeta topos} is the topos of the form $\sh(\mc C,J)$ for some extensive category $\mc C$ equipped with the extensive topology. If $\mc C$ is in fact \emph{lextensive}, then $\sh(\mc C,J)$ is finitary disjunctive.
\end{exa}

\begin{thm}
    Any (finitary) disjunctive topos is formally (finitary) disjunctive, and it admits an $\mc H_{\mr{inn}}$- ($\mc H_{\mr{pure}}$-)presentation.
\end{thm}
\begin{proof}
    By Proposition~\ref{prop:genericproof}, again consider an innocent map $f : \mc E \to \mc F$ with $x : \mc E \to \sh(\mc C,J)$ into a disjunctive topos. Let $j : \sh(\mc C,J) \hook \psh(\mc C)$ be the inclusion and $y = j^*\yo$. Then $f_*x^*y$ is $J$-continuous iff whenever we have a $J$-covering $\stt{X_i \inj X}_{i\in I}$ in $\mc C$, it sends it to a coproduct diagram. But this indeed holds since by construction $\coprod_{i\in I}yX_i \cong yX$, and $f_*$ by definition preserves coproducts. The finitary case is completely similar.
\end{proof}

\subsection{Regular Logic} \label{regular}

To characterise the regular fragment, we introduce the following family of geometric morphisms:

\begin{defn}\label{defmatte}
    We say a geometric morphism $f$ is \emph{matte} if $f_*$ preserves epimorphisms. The class of matte geometric morphisms will be denoted as $\mc H_{\mr{matte}}$.
\end{defn}

In fact, being matte is equivalent to a seemingly stronger condition:

\begin{lem}\label{lem:inductiveiffeffective}
    $f : \mc E \to \mc F$ is matte iff $f_*$ preserves effective quotients.
\end{lem}
\begin{proof}
    This follows from the fact that every epi in a topos is effective, i.e. it is the coequaliser of its kernel pair, and the fact that $f_*$ always preserves pullbacks.
\end{proof}

\begin{rem}
    The notion of a matte geometric morphism has occurred in the literature, though without a name. For instance,~\cite[Prop. 3.2]{marmolejo2005locale} has characterised such geometric morphisms between localic topoi. As another example, it appears in~\cite[Prop. 3.5]{blechschmidt2018flabby} showing that the direct image of a matte geometric morphism preserves the so-called flabby objects.
\end{rem}

\begin{defn}
    We say a topos $\mc X$ is \emph{formally regular} if it is weakly right Kan injective w.r.t. matte geometric morphisms.
\end{defn}

By looking at canonical examples of matte geometric morphisms, we can again infer that certain limits exist in the category of points. Firstly, notice that $\emptyset \hook \mc E$ for any topos $\mc E$ is matte, since the initial topos $\emptyset$ as a category is the terminal category, and all maps in it are already isomorphisms. This implies that any formally regular topos $\mc X$ is totally connected, by Theorem~\ref{thm:totallyconnected}. More generally, its category of points have \emph{all} products:

\begin{lem} \label{products}
    For any formally regular topos $\mc X$, its category of points $\ms{pt}(\mc X)$ has all products.
\end{lem}
\begin{proof}
    This follows from the fact that for any set $I$, the geometric morphism
    \[ I \otimes \Set \simeq [I,\Set] \to \Set \]
    is matte, since $I$-indexed product of surjections is again a surjection in $\Set$.
\end{proof}

\begin{rem}\label{rem:choicereg}
    Notice that the above proof does rely on the fact that choice holds in $\Set$, so that arbitrary products of epimorphisms is again an epi.
\end{rem}

On the logical side, the story is more familiar. We say a site $(\mc C,J)$ is \emph{regular} if $\mc C$ is a regular category with $J$ its regular topology.

\begin{defn}
    We call a topos \emph{regular} if it is of the form $\sh(\mc C,J)$ for a regular site. Equivalently, it is the classifying topos of a regular theory.
\end{defn}

Our goal is to show that regular topoi are also formally regular. This is again immediate by Proposition~\ref{prop:genericproof}: 

\begin{thm}
    A regular topos is formally regular, and admits $\mc H_{\mr{matte}}$-presentation.
\end{thm}
\begin{proof}
    Let $f : \mc E \to \mc F$ be matte, and $x : \mc E \to \mc X$ a map into a regular topos. Let $j : \mc X \simeq \sh(\mc C,J) \hook \psh(\mc C)$ be a site presentation of $\mc X$ with $(\mc C,J)$ a regular site. By Proposition~\ref{prop:genericproof}, it suffices for $f_*x^*y$ with $y = j^*\yo$ to be $J$-continuous. For any $J$-cover $q : Y \to X$ in $\mc C$, we know that $yq$ is epi, thus $f_*x^*yq$ is also an epi since $f$ is by assumption matte. Hence, $\mc X$ is formally regular, and $\mc X \hook \psh(\mc C)$ is an $\mc H_{\mr{matte}}$-presentation of $\mc X$.
\end{proof}

\subsection{Coherent Logic} \label{coherent}

As our last example, we recall the results obtained in~\cite{diliberti2022geometry}, which captures coherent logic via right Kan injectivity. The relevant notion of geometric morphisms is the following:

\begin{defn}\label{defflat}
    We say a geometric morphism $f$ is \emph{flat} if it is both pure and matte, i.e. it preserves both finite coproducts and epimorphisms.
\end{defn}

\begin{defn}
    We say a topos $\mc X$ is formally coherent if it is weakly right Kan injective w.r.t. flat maps.
\end{defn}

Of course, as shown in~\cite{diliberti2022geometry}, any coherent topos will be formally coherent:

\begin{thm}\label{thm:cohformal}
    A coherent topos is formally coherent, and also admits an $\mc H_{\mr{flat}}$-presentation.
\end{thm}
\begin{proof}
    Consider a flat geometric morphism $f : \mc E \to \mc F$ with $x : \mc E \to \sh(\mc C,J)$ where $(\mc C,J)$ is a pretopos with the coherent topology. Let $j : \sh(\mc C,J) \hook \psh(\mc C)$ be the canonical embedding and $y = j^*\yo$. Again by Proposition~\ref{prop:genericproof}, it suffices to show $f_*x^*y$ is always $J$-continuous. In other words, we need to show $f_*x^*y$ is coherent. This indeed holds since $y$ is coherent, and $f_*$ by assumption preserves coproducts and effective epis, hence is again coherent. 
\end{proof}

\begin{rem} \label{jeremie}
    In the first draft of~\cite{diliberti2022geometry}, the notion of flat maps is defined as those geometric morphisms $f$ whose direct image $f_*$ preserves finite colimits (as opposed to finite coproducts and epimorphisms). From the above proof one can see that this is indeed a bit stronger than what we need to show Theorem~\ref{thm:cohformal}. The first author will soon update the ArXiv entry of his paper to address this mismatch in definition.
\end{rem}

At this point, let us recall the original motivation of investigating Kan injectivity in $\Topoi$ as a characterisation of coherent logic presented in~\cite{diliberti2022geometry}. As we have demonstrated in previous sections, a right Kan injectivity condition in $\Topoi$ represents a general notion of limits in the category of generalised points of a topoi. The notion of limit as specified by flat geometric morphisms though is more subtle. The observation, already in some sense contained in~\cite{marmolejo2005locale} and certainly expanded in~\cite{diliberti2022geometry}, is that the \emph{ultraproduct structure} for the points of coherent topoi is encoded this way. 

\begin{rem}[Again on prescribing universals] \label{prescribinguniversals}
To repeat a part of discussion in~\cite{diliberti2022geometry}, let $I$ be a discrete set, and consider the inclusion $I \hook \beta I$, where $\beta I$ is the Stone-\v{C}ech compatification of $I$, viz. the space of ultrafilters on $I$. This induces a flat embedding $i : \sh(I) \simeq \Set^I \hook \sh(\beta I)$. Given any formally coherent topos $\mc X$, we may identify the right Kan extension of $\mc X$ against $i$ as a formal ultraproduct in $\ms{pt}(\mc X)$,
\[
\begin{tikzcd}
    \Set^I \ar[d, hook, "i"'] \ar[r, "x"] & \mc X \\ 
    \sh(\beta I) \ar[ur, dashed, "\ran_ix := \int xd(-)"']
\end{tikzcd}
\]
In other words, for any $I$-indexed family of points $x : I \to \ms{pt}(\mc X)$ and for any ultrafilter $\mu\in\beta I$, viewed as a point $\mu : \Set \to \sh(\beta I)$, we may define the ultraproduct of $x$ along $\mu$ as the composite,
\[ \int xd\mu := \ran_ix \circ \mu. \]
When $\mc X$ is a coherent topos, then the above computes literally the ultraproducts in $\ms{pt}(\mc X)$ in the sense of model theory.

In~\cite{diliberti2022geometry}, the class of maps $\stt{\Set^I \hook \sh(\beta I)}_{I\in\Set}$ are called $\beta$-maps, and we may denote them as $\mc H_{\beta}$. Let $\ms{Ult}$ be the 2-category of ultracategories and ultrafunctors in the sense of~\cite{makkai1987stone}; see also~\cite{lurieultracategories}. In fact, it is shown in~\cite[Prop. 3.1.3]{diliberti2022geometry} that taking the points of a topos produces a 2-functor
\[ \RInj(\mc H_\beta) \to \ms{Ult}, \]
Since for coherent topoi, by Theorem~\ref{thm:cohformal} they all admit an $\mc H_{\mr{flat}}$-presentation hence are indeed right Kan injective w.r.t. $\beta$-maps, this explains the ultrastructure on the category of points of a coherent topos.
\end{rem}

\begin{rem}[Semantic prescriptions, first encounter]\label{rem:semanticcoherent}
    At this point, we reflect on our approach of characterising logic via Kan injectivity. From the above discussion, it is clear that the 2-category $\WRInj(\mc H_{\mr{flat}})$ has objects those topoi whose category of points admits a generalised form of ultraproducts, and morphisms those geometric morphisms preserving such generalised ultraproducts. From this perspective, $\mc H_{\mr{flat}}$ should be understood as a \emph{semantic prescription}, which forces the generalised category of points of topoi in $\WRInj(\mc H_{\mr{flat}})$ to behave as expected. We will return to a more detailed discussion of this perspective in Section~\ref{sec:logic}. 
\end{rem}

At the end of this section, let us stress that there are reasons to believe the right notion of semantic prescription for the coherent fragment is the class of flat maps, as opposed to the more restrictive collection of $\beta$-maps. In particular, we address one question raised in~\cite{diliberti2022geometry}, which is the comparison between the right Kan injectivity classes w.r.t. $\mc H_{\beta}$ and $\mc H_{\mr{flat}}$:

\begin{prop}\label{prop:betanotflat}
    For any topos $\mc X$ without any point, $\mc X \in \RInj(\mc H_{\beta})$. In particular, $\RInj(\mc H_{\mr{flat}})$ is strictly included in $\RInj(\mc H_{\beta})$.
\end{prop}
\begin{proof}
    Consider a set $I$. If $I$ is empty, then the Stone-\v{C}ech compatification is given by $\emptyset \hook \beta\emptyset \cong \emptyset$, which induces an equivalence between their sheaf topoi. Hence, $\mc X$ must be right Kan injective w.r.t. this map. If $I$ is not empty, then by assumption there are no geometric morphisms $\Set^I \to \mc X$, which implies that $\mc X$ is trivially right Kan injective w.r.t. $\Set^I \to \sh(\beta I)$. It is furthermore shown in~\cite[Prop. 1.3.6]{diliberti2022geometry} that any $\mc X\in\RInj(\mc H_{\mr{flat}})$ has enough points, hence the two right Kan injectivity classes are distinct.
\end{proof}

\section{Logic and Doctrine}\label{sec:logic}

With enough examples at hand, in this section we will investigate the formal consequences of \emph{identifying} a fragment of logic as a weak right Kan injectivity class in the 2-category $\Topoi$. We will thus impose the following definition:

\begin{defn}
    A \textit{(fragment of geometric) logic} is just a class of geometric morphisms $\Hcal$. A topos $\mc X$ is formally in the logic $\mc H$ if it belongs to the weak right Kan injectivity class $\WRInj(\mc H)$. 
\end{defn}

\begin{rem}
    Every subsection of Section~\ref{sec:logic} established and discusses a fragment of geometric logic according to our definition.
\end{rem} 

\begin{rem}[Semantic prescription, second encounter] \label{prescriptiononelasttime}
    Similar to the discussion for the coherent fragment in Remark~\ref{rem:semanticcoherent}, a general fragment $\mc H$ in our sense specifies a fragment of logic via imposing semantic constraints on its generalised category of points. The 2-category $\WRInj(\mc H)$ will thus consist of those topoi (theories) whose generalised points behave as specified by $\mc H$, and morphisms preserving these structures. Without any specific commitment to its logical meaning, we have already used this point of view many times already throughout the paper. The simplest examples were given in \Cref{sec:logicshowoff} when we wanted to \textit{prescribe} our categories of models of have \textit{connected limits} (\ref{connected limits}) or \textit{products} (\ref{products}). The case of \textit{ultraproducts} that we have discussed at the end of the previous section is more sophisticated, but of the same nature.  
\end{rem}

\begin{rem}[Three aspects of logic]
    As mentioned in the introduction, though we define a fragment of logic as the semantic prescription $\mc H$, the behaviour of the fragment is really reflected through the combination of $\mc H$, the objects in $\WRInj(\mc H)$, and the lax-idempotent (relative) pseudomonad $\ms T^{\Hcal}$ which will be constructed later in this section.
\end{rem}

\begin{rem}[A comment on propositional logic]\label{rem:prop}
The above definition of logic is crafted around \textit{predicate} logic (as opposed to propositional). Indeed, because free topoi are weakly injective with respect to all geometric morphisms, every free topos will belong to every fragment of geometric logic with our definition. This matches our expectation: of course an essentially algebraic theory is also a coherent theory where we are not using the whole expressive power of coherent logic. Yet, for this reason we are ruling out what we called \textit{propositional logic} in \Cref{propositionallogic}, which could only be describe by an injectivity class (as opposed to weak (!) injectivity class). This is a descriptive choice because we are mostly interested in predicate logic and other choices could have been possible.
\end{rem}

\begin{rem}[Limitations of our framework]\label{rem:limitation}
    The fragments in our framework are of course \emph{property like}. In particular, our framework does not cover any modal logic. The fragments are also fragments of \emph{geometric logic}, so that, say, first order intuitionistic logic is not covered by our framework either.
\end{rem}

\begin{exa}[An exotic example: integral geometric logic]\label{exa:complementedfragment}
    The \emph{single} geometric morphism $\emptyset \hook \Set$ as discussed in Section~\ref{subsec:termandtot} will now also be a fragment of geometric logic according to the above definition. We will now use $\mc H_{\mr{cl}}$ to denote this fragment, and call the corresponding logic \emph{integral geometric logic}. We have completely characterised the 2-category $\WRInj(\mc H_{\mr{cl}})$, and it would be interesting to work out a syntactic description of this fragment.
\end{exa}

Let us fix a logic $\mc H$ from the start. In the following texts we will see, surprisingly, how such a ``na\"ive'' definition connects the story about \emph{doctrines}, \emph{syntactic categories}, \emph{classifying topoi}, and \emph{conceptual completeness}.

\subsection{Syntactic categories and syntactic sites}

\begin{constr}[The syntactic category] \label{syncat}
    We shall define a $2$-functor 
    \[ \mathsf{Syn}^{\mathcal{H}}:\WRInj(\mathcal{H})^{\text{op}} \to \LEX, \]
    mapping $\mathcal{X} \mapsto \WRInj(\mathcal{H})(\mathcal{X}, \mathsf{Set}[\mathbb{O}])$. Notice that since  $\mathsf{Set}[\mathbb{O}]$ is the object classifier and $\WRInj(\mc H)$ is locally fully faithful, $\mathsf{Syn}^{\mathcal{H}}(\mc X)$ is a full subcategory of $\mc X$.
\end{constr}

We shall prove soon (\Cref{synislex}) that the construction above indeed lands in $\LEX$. But before doing that, we shall compute such association in a couple of examples to get used to its behaviour.

\begin{rem}
   Observe that in full generality it is hard to provide a complete classification of the morphisms of injectives with respect to a general class $\Hcal$. Hence, computing $\mathsf{Syn}^{\mathcal{H}}(\mc X)$ explicitly can be very challenging. In \Cref{conceptual} we will see that such a complete description is indeed a form of conceptual completeness for our logic.
\end{rem}

\begin{exa} \label{examplesofsyns}
    For the reason above, and since we just gave the definition, we can't leave the realm of very simple examples:
    \begin{itemize}
        \item[$\mathcal{H}_{\text{eth}}$] Let $\mc X$ be any topos. Since $\WRInj(\mc H_{\mr{eth}}) = \Topoi$ by Proposition~\ref{chargeom}, every geometric morphism $\mc X \to \Set[\mbb O]$ will lie in $\WRInj(\mc H_{\mr{eth}})$. It follows that $\mathsf{Syn}^{\mathcal{H_{\mr{eth}}}}(\mc X)$ is $\mc X$ itself, and $\ms{Syn}^{\mc H_{\mr{eth}}}$ is the forgetful functor
        \[ \ms U : \Topoi\op \to \LEX, \]
        now mapping a geometric morphism to its inverse image.
        \item[$\mathcal{H}_{\text{all}}$] On the opposite side of the spectrum, for $\mc X \in \WRInj(\mc H_{\mr{all}})$, a geometric morphism $\mc X \to \Set[\mbb O]$ belongs to $\WRInj(\mc H_{\mr{all}})$ iff it is ethereal by Lemma~\ref{lem:adjointiffallinj}. In particular, for a free topos $\psh(\mc C)$, by Corollary~\ref{class:essalg} $\mathsf{Syn}^{\mathcal{H}}(\psh(\mc C))$ coincides precisely with the full subcategory of representables, a.k.a. $\mc C$ itself.
        \item[$\mc H_{\mr{cl}}$] Let $\mc X$ be a totally connected topos, with $\Pi : \mc X \to \Set$ the inverse image of the dense point, which is the left adjoint of the constant functor $\Delta : \Set \to \mc X$. By Corollary~\ref{cor:classtot}, a simple computation shows $\ms{Syn}^{\mc H_{\mr{cl}}}(\mc X)$ consists of those objects $X$ where $\Pi X \cong 1$.
        \item[$\mathcal{H}$] For each one of the examples we gave in the previous section, $\ms{Syn}^{\mc H}(\psh(\mc C))$ will land somewhere in between the representables and the whole topos.
    \end{itemize}
\end{exa}

\begin{prop} \label{synislex}
    For any $\mc X$ in $\WRInj(\mc H)$, the category $\ms{Syn}^{\mc H}(\mc X)$ is closed under finite limits in $\mc X$.
\end{prop}
\begin{proof}
    Consider a map $f : \mc E \to \mc F$ in $\mc H$, and any $x : \mc E \to \mc X$. Suppose we have a finite diagram $(X_i)_{i\in I}$ of objects in $\mc X$, such that for each $X_i : \mc X \to \ms{Set}[\mbb O]$, $X_i$ is right Kan injective w.r.t. $\mc H$. We show that $(\lt_{i\in I}X_i) \circ \ran_fx$ is the right Kan extension of $(\lt_{i\in I}X_i)\circ x$ along $f$. This follows from the following computation: For any $F$ in $\mc F$,
    \begin{align*}
        &\Topoi(\mc F,\Set[\mbb O])(F,(\lt_{i\in I}X_i)\circ\ran_fx) \\
        \cong& \mc F(F,(\ran_fx)^*\lt_{i\in I}X_i) \\
        \cong& \mc F(F,\lt_{i\in I}(\ran_fx)^*X_i) \\
        \cong& \lt_{i\in I} \Topoi(\mc F,\Set[\mbb O])(F,X_i \circ \ran_fx) \\
        \cong& \lt_{i\in I} \Topoi(\mc F,\Set[\mbb O])(F,\ran_f(X_ix)) \\
        \cong& \lt_{i\in I} \Topoi(\mc E,\Set[\mbb O])(Ff,X_ix) \\
        \cong& \Topoi(\mc E,\Set[\mbb O])(Ff,(\lt_{i\in I}X_i) \circ x)
    \end{align*}
    Most steps follow formally from the fact that inverse images preserves finite limits, and the fact that $\Set[\mbb O]$ is the object classifier. Crucially, the third isomorphism holds by the assumption that each $X_i$ preserves the right Kan extension along $f$. Thus, it follows that the composite $\lt_{i\in I}X_i \circ \ran_fx$ satisfies the universal property of being the right Kan extension, hence the finite limit $\lt_{i\in I}X_i$ also preserves this right Kan extension.
\end{proof}


\begin{constr}[The syntactic site] \label{synsite}
    Let $\mc X$ be a topos in $\WRInj(\mc H)$. We shall now equip its syntactic category $\mathsf{Syn}^\Hcal(\mc X)$ with a canonical site structure. This provides us with a lifting of the syntactic construction to the $2$-category of (possibly large) sites.
    \[\begin{tikzcd}
    	& {\mathsf{SITES}} \\
    	{\WRInj(\Hcal)^{\text{op}}} & \LEX
    	\arrow[from=1-2, to=2-2]
    	\arrow["{\mathbb{Syn}^\Hcal}", dashed, from=2-1, to=1-2]
    	\arrow["{\mathsf{Syn}^\Hcal}"', from=2-1, to=2-2]
    \end{tikzcd}\]

    So, we shall define a topology $D_{\mc X}^{\Hcal}$ on $\mathsf{Syn}^\Hcal(\mc X)$ as follows, which will provide us with the lifting $\mathbb{Syn}^\Hcal(\mc X) = (\mathsf{Syn}^\Hcal(\mc X), D_{\mc X}^{\Hcal})$. We know from the previous discussion that $\mathsf{Syn}^\Hcal(\mc X)$ can be understood as a full subcategory of $\mc X$, with inclusion $i: \mathsf{Syn}^\Hcal(\mc X) \to \mc X$. Of course, this embedding is witnessing a pseudonatural transformation as below,
    \[\begin{tikzcd}
    	{\WRInj(\Hcal)^{\text{op}}} & \LEX
    	\arrow[""{name=0, anchor=center, inner sep=0}, "{{\mathsf{Syn}^\Hcal}}", shift left=3, from=1-1, to=1-2]
    	\arrow[""{name=1, anchor=center, inner sep=0}, "{{\mathsf{U}}}"', shift right=3, from=1-1, to=1-2]
    	\arrow["i", shorten <=1pt, shorten >=1pt, Rightarrow, 2tail, from=0, to=1]
    \end{tikzcd}\]

    Now we use the fact that the forgetful functor from sites to lex categories is $2$-topological in the sense of \cite[Sec 2.1]{di2024sketches}, thus we can pullback the canonical site structure that $\mc X$ has over the syntactic category; see also \cite[C2.3]{elephant2}. Hence, we define
    \[ D_{\mc X}^{\Hcal} = i^*(J_{\mc X}^{\text{can}}). \]
    It is very easy to see that this construction is functorial because inverse images preserve canonical covers.
\end{constr}

\subsubsection{Bounded Logic}

Both the syntactic category and the syntactic site of a topos in a fragment of geometric logic may be large categories. In practice though, we do not have that many examples of \textit{large} syntactic categories, and for many examples we can prove or we conjecture that the syntactic category is indeed small. Let us give two additional definitions that are useful in this taxonomy.

\begin{defn}[Bounded logic]
   A logic $\mathcal{H}$ will be called \textit{bounded} if its syntactic functor lands in small (lex) categories/sites.  We say that a logic $\mathcal{H}$ is \textit{Morita-bounded} if its syntactic functor lands in Morita-small sites\footnote{Recall that a (lex) site is Morita-small if it is Morita equivalent to a small (lex) site.}. Bounded fragments are Morita-bounded, clearly.
\end{defn}

Bounded fragments are easier to treat as they eliminate several size issues that will have technical repercussions in the next sections. Besides the technical relevance, there is some heuristic difference between bounded and unbounded fragments, even if at the moment we cannot provide a full treatment of them. Bounded fragments are \textit{bounded} in the sense that the size of allowed disjunctions is uniformly bounded by a cardinal. At this stage, we can provide two polarised examples of this phenomenon.
 
\begin{exa}
 We saw in \Cref{examplesofsyns} that $\Hcal_{\text{eth}}$ is bounded. This matches our intuition as essentially algebraic logic allows for no disjunctions at all. \Cref{prop:synofcoh} shows that the logic associated to $\Hcal_{\text{flat}}$ is also bounded, but the proof hinges on Makkai's conceptual completeness. So, while in practice quite a number of fragments are bounded, to actually show that they are bounded seems very technical at this stage of maturity of the theory. $\Hcal_{\text{all}}$ is Morita-bounded\footnote{Every topos -- seen as a site with canonical topology -- is Morita-equivalent to a small site of definition.} but it is not bounded. This example is on the opposite side of the spectrum, indeed for full fledged geometric logic one cannot bound the size of disjunctions. We have no example of a logic that is not Morita-bounded and we believe they don't exist.
\end{exa}

\subsection{From Logic to Doctrine ($\Hcal \rightsquigarrow \mathsf{T}^{\Hcal}$)}
The aim of this section is to connect logic to the so-called \textit{doctrine} introduced (independently) by Kock and Zöberlein \cite{kock1977doctrines,zoberlein1976doctrines}. Yet, for our purposes, we need a specific flavour of such notion. On one side we need to generalise to \textit{relative} lax-idempotent pseudomonads \cite{fiore2018relative}, and on the other side, we are only interested in some doctrines defined on the $2$-category of left exact categories. Hence, for the sake of this paper, we shall use the following definition of doctrine.

\begin{defn}[Doctrine]
    A doctrine $(\ms{T}, \sigma)$ is a lax-idempotent relative pseudomonad $\mathsf{T}$ over the inclusion $j: \mathsf{lex} \to \LEX$ equipped with a locally fully faithful pseudonatural transformation\footnote{It would be more appropriate to require this pseudonatural transformation to be a morphism of relative pseudomonads, possibly following \cite[Def. 61 and Lem. 62]{walker2019distributive}. Unfortunately this theory has not been developed yet and is falls outside of the scope of this paper. Later in \Cref{lessinterestingly} we will \textit{in practice} show that we deal with a morphism of relative pseudomonads (if such notion existed).} $\sigma: \ms T \Rightarrow \psh$ between $\ms T$ and the presheaf construction $\psh$. A doctrine is \textit{bounded} if the relative pseudomonad preserves smallness. 
\end{defn}




\begin{rem}[A \textit{reasonable} definition of doctrine] \label{constr:fromlotodoc}
    The notion of \textit{doctrine} defined above matches the tradition of categorical logic. Indeed all the $2$-categories of \textit{theories} associated to fragments of predicate logic can be specified by a doctrine in the sense above (see the introduction of \cite{di2025bi} and its last subsection). It is also useful to see \cite[Example 1.1.8]{di2024sketches}. The $2$-category of -- say -- small coherent categories is indeed the category of algebras for a lax-idempotent pseudomonad on $\lex$ sitting inside the presheaf construction (\cite[Sec. 6]{di2025bi} and \cite[Sec. 5.6]{lex}). Relative pseudomonads are one possible solution to handle size issue: the $2$-category of -- say -- infinitary pretopoi is indeed the 2-category of algebras for the relative lax-idempotent pseudomonad on $\lex$ given by the presheaf construction \cite[Prop. 2.5]{lex}.
\end{rem}

\begin{constr}[The Doctrine associated to a Logic]
    For $\mathcal{H}$ a logic, consider the composition below. 
    \[\begin{tikzcd}
    	\mathsf{lex} & {\WRInj(\mathcal{H})^{\text{op}}} & \LEX
    	\arrow["{\psh}"', from=1-1, to=1-2]
    	\arrow["{\mathsf{T}^{\Hcal}}"{description}, curve={height=-30pt}, dashed, from=1-1, to=1-3]
    	\arrow["{\mathsf{Syn}^{\mathcal{H}}}"', from=1-2, to=1-3]
    \end{tikzcd}\]
    We shall now prove that $\mathsf{T}^{\Hcal}$ is indeed a lax-idempotent relative pseudomonad. When $\mathcal{H}$ is bounded, this construction will produce a bounded doctrine.
\end{constr}

\subsubsection{$\ms T^\mathcal{H}$ is indeed a doctrine}

The proof of this result will require some preliminary lemmas and the introduction of some terminology. When $\mc D$ is a \emph{large} lex category, we will use $\psh(\mc D)$ to denote the category of \emph{small} presheaves on $\mc D$, i.e. those presheaves arising as small colimits of representables. $\psh(\mc D)$ has the correct universal property, which is the free cocompletion of $\mc D$ under all small colimits.

\begin{term}
    Given $F : \mc C \to \mc D$ in $\LEX$ with $\mc C$ small, it induces an adjoint triple between the corresponding (small) presheaf categories as follows,
    \[\begin{tikzcd}
    	{\psh(\mc C)} & {\psh(\mc D)}
    	\arrow[""{name=0, anchor=center, inner sep=0}, "{F_!}", curve={height=-12pt}, from=1-1, to=1-2]
    	\arrow[""{name=1, anchor=center, inner sep=0}, "{F_*}"', curve={height=12pt}, from=1-1, to=1-2]
    	\arrow[""{name=2, anchor=center, inner sep=0}, "{F^*}"{description}, from=1-2, to=1-1]
    	\arrow["\dashv"{anchor=center, rotate=-91}, draw=none, from=0, to=2]
    	\arrow["\dashv"{anchor=center, rotate=-89}, draw=none, from=2, to=1]
    \end{tikzcd}\]
    The left adjoint $F_!$ exists since $\psh(\mc D)$ is cocomplete, hence has pointwise left Kan extension; it is also left exact because $F$ is. The right adjoint $F_*$ exists follows from the special adjoint functor theorem (cf.~\cite[Prop. 2.2.10]{di2024formal}). We still use the notation of geometric morphism to denote them $\Sigma F : \psh(\mc C) \leftrightarrows \psh(\mc D) : \mc PF$, where
    \[ (\mc PF)_* = F^* = (\Sigma F)^*, \quad (\mc PF)^* = F_!, \quad (\Sigma F)_* = F_*. \]
\end{term}

\begin{lem}
    For any $\mc C$ in $\ms{lex}$, we have a fully faithful unit $\eta : \mc C \hook \ms T^{\mc H}(\mc C)$ in $\LEX$.
\end{lem}
\begin{proof}
    For any $\mc C$ in $\ms{lex}$, it is easy to see we have a fully faithful unit $\eta : \mc C \hook \ms T^{\mc H}(\mc C)$, since the representable $\yo_c : \psh(\mc C) \to \Set[\mbb O]$ is a right adjoint in $\Topoi$ for any $c\in\mc C$, thus preserves all right Kan extensions. Fully faithfulness follows from Yoneda. It is left exact by construction.
\end{proof}

Now given $F : \mc C \to \ms T^{\mc H}(\mc D)$, to construct the extension $\qsi F : \ms T^{\mc H}(\mc C) \to \ms T^{\mc H}(\mc D)$ it would be equivalent to construct a geometric morphism $\psh(\mc D) \to \psh(\mc C)$ in $\WRInj(\mc H)$, since then we can then take $\qsi F$ to be given by pre-composition w.r.t. this map,
\[ \qsi F : \ms T^{\mc H}(\mc C) \simeq \WRInj(\psh(\mc C),\Set[\mbb O]) \to \WRInj(\psh(\mc D),\Set[\mbb O]) \simeq \ms T^{\mc H}(\mc D). \]
The only way to build this is as follows,
\[\begin{tikzcd}
	& {\mc D} \ar[drr, draw=none, "\mapsto"description] &&& {\psh(\mc D)} \\
	{\mc C} & {\ms T^{\mc H}(\mc D)} && {\psh(\mc C)} & {\psh(\ms T^{\mc H}(\mc D))}
	\arrow["\eta", hook, from=1-2, to=2-2]
	\arrow["F"', from=2-1, to=2-2]
	\arrow["{\mc PF}", from=2-5, to=2-4]
	\arrow["{\Sigma\eta}", hook, from=1-5, to=2-5]
\end{tikzcd}\]

\begin{lem}
    The composite geometric morphism $\mc PF \circ \Sigma \eta : \psh(\mc C) \to \psh(\mc C)$ always lies in $\WRInj(\mc H)$.
\end{lem}
\begin{proof}
    Let $f : \mc E \to \mc F$ be a map in $\mc H$ and $x : \mc E \to \psh(\mc D)$. Since the right Kan extensions for free topoi are pointwise as shown in Example~\ref{exm:ranfree}, we have
    \[ (\mc PF \circ \Sigma\eta \circ \ran_fx)^* \cong (\ran_fx)^* \circ \eta^* \circ F_!. \]
    Now we observe that $\eta^*$ is itself a left Kan extension of $\sigma : \ms T^{\mc H}(\mc D) \hook \psh(\mc D)$,
    \[
    \begin{tikzcd}
        \ms T^{\mc H}(\mc D) \ar[r, hook, "\sigma"] \ar[d, "\yo"'] & \psh(\mc D) \\
        \psh(\ms T^{\mc H}(\mc D)) \ar[ur, "\lan_{\yo}\sigma\cong\eta^*"']
    \end{tikzcd}
    \]
    Since $(\ran_fx)^*$ is a left adjoint, we then have
    \[ (\mc PF \circ \Sigma\eta \circ \ran_fx)^* \cong (\lan_{\yo}(\ran_fx)^*\sigma) \circ F_!. \]
    Since by assumption, objects lying in the image of $\sigma$ preserves right Kan extension along maps in $\mc H$, and by Example~\ref{exa:directimagerightKan} right Kan extensions for $\Set[\mbb O]$ are computed by direct images, we can perform the computation below, 
    \[ \lan_{\yo}(\ran_fx)^*\sigma \cong \lan_{\yo}f_*x^*\sigma \cong \lan_{\yo}f_*x^*\eta^*\yo. \]
    It follows that
    \begin{align*}
   (\lan_{\yo}(\ran_fx)^*\sigma) \circ F_!     \cong &(\lan_{\yo}f_*x^*\eta^*\yo) \circ F_! \\
        \cong &\lan_{\yo}f_*x^*\eta^*\yo F \\ 
        \cong &\lan_{\yo}f_*x^*F_!\yo \\ 
        \cong &(\ran_f\mc PF \circ \Sigma\eta \circ x)^*.
    \end{align*}
    Hence, $\mc PF \circ \Sigma\eta$ lies in $\WRInj(\mc H)$.
\end{proof}



\begin{thm}\label{thm:laxidempotentmonad}
    $\ms T^{\mc H}$ is a lax-idempotent relative pseudomonad.
\end{thm}
\begin{proof}
    \emph{Extension}: For any $F : \mc C \to \ms T^{\mc H}(\mc D)$, by construction we have $\qsi F \circ \eta_{\mc C} \cong F$. 
    
    \emph{Unit}: By construction the extension $\qsi\eta$ of $\eta : \mc C \hook \ms T^{\mc H}(\mc C)$ corresponds to the geometric morphism $\mc P\eta \circ \Sigma\eta \cong 1$. 
    
    \emph{Multiplication}: Suppose we have $F : \mc C \to \ms T^{\mc H}(\mc D)$ and $G : \mc D \to \ms T^{\mc H}\mc E$. By construction of $\qsi G$, the following diagram commutes,
    \[
    \begin{tikzcd}
        \psh(\mc E) \ar[d, "\Sigma\eta_{\mc E}"'] \ar[r, "\mc PG \circ \Sigma\eta_{\mc E}"] & \psh(\mc D) \ar[d, hook, "\Sigma\eta_{\mc D}"] \\ 
        \psh(\ms T^{\mc H}\mc E) \ar[r, "\mc P\qsi G"'] & \psh(\ms T^{\mc H}(\mc D))
    \end{tikzcd}
    \]
    Now the extension $\qsi{\qsi G F}$ corresponds to the map
    \[ \mc P(\qsi GF) \circ \Sigma\eta_{\mc E} \cong \mc PF \circ \mc P\qsi G \circ \Sigma\eta_{\mc E} \cong \mc PF \circ \Sigma\eta_{\mc D} \circ \mc PG \circ \Sigma\eta_{\mc E}, \]
    which is exactly the map corresponds to the composition $\qsi G \circ \qsi F$. Verifying they satisfy the coherent axioms is routine.
\end{proof}

Finally, the following result shows we in practice have a morphism of lax-idempotent relative pseudomonad from $\ms T^{\mc H}$ to $\psh$: 

\begin{prop}\label{prop:morphismofpsalg}
    The inclusion $\sigma : \ms T^{\mc H}(\mc C) \hook \psh(\mc C)$ for any $\mc C$ in $\lex$ induces a pseudonatural transformation $\sigma : \ms T^{\mc H} \nt \psh$, which is computed as a left Kan extension in $\LEX$ as below,
    \[
    \begin{tikzcd}
        \mc C \ar[d, hook, "\eta"'] \ar[r, hook, "\yo"] & \psh(\mc C) \\ 
        \ms T^{\mc H}(\mc C) \ar[ur, hook, "\sigma \cong \lan_{\eta}\yo"']
    \end{tikzcd}
    \]
\end{prop}
\begin{proof}
    For any $X\in\ms T^{\mc H}(\mc C)$, notice that by definition the two comma categories $\eta \downarrow X$ and $\yo \downarrow \sigma X$ coincides. Hence, by Yoneda we have
    \[ \sigma X \cong \ct_{C \in \yo\downarrow \sigma X}\yo_C \cong \ct_{C\in\eta\downarrow X} \yo_C \cong (\lan_\eta\yo) X. \]
    The final isomorphism holds because $\lan_\eta\yo$ is computed pointwise in $\CAT$.
\end{proof}

\subsection{Algebras for $\mathsf{T}^{\Hcal}$ and classifying topoi}

In this subsection we discuss some important properties for the $2$-category of algebras of $\mathsf{T}^{\Hcal}$. More specifically for us, we use the definition of \emph{lax-idempotent pseudoalgebras} for a lax-idempotent relative pseudomonad introduced in~\cite[Def. 5.1 and Prop. 5.14]{arkor2025bicategories}. This is our default notion, thus by algebra we will always mean a lax-idempotent pseudo one. Concretely, an algebra structure on a (possibly large) lex category $\mc C$ is essentially unique, and it exists iff for any $f : \mc A \to \mc C$ in $\LEX$ from a small lex category, the left Kan extension $a^f$ in $\LEX$ exists,
\[\begin{tikzcd}
	{\mc A} & {\mc C} \\
	{\mathsf{T}^{\Hcal}(\mc A)}
	\arrow["f", from=1-1, to=1-2]
	\arrow["{{\eta}}"', from=1-1, to=2-1]
	\arrow["{a^f}"', dashed, from=2-1, to=1-2]
\end{tikzcd}\]
which satisfies the following conditions:
\begin{itemize}
    \item The unit $1 \cong a^f\eta$ is invertible;
    \item These left Kan extensions are stable in the sense that for any $g : \mc B \to \ms T^{\Hcal}(\mc A)$ in $\LEX$ from a small $\mc B$, we have $a^f\qsi g \cong a^{a^fg}$ as shown below,
    \[\begin{tikzcd}
    	& {\mc A} & {\mc C} \\
    	{\mc B} & {\mathsf{T}^{\Hcal}(\mc A)} \\
    	{\ms T^{\Hcal}(\mc B)}
    	\arrow["f", from=1-2, to=1-3]
    	\arrow["{{\eta}}"', from=1-2, to=2-2]
    	\arrow["g", from=2-1, to=2-2]
    	\arrow["\eta"', from=2-1, to=3-1]
    	\arrow["{a^f}"description, dashed, from=2-2, to=1-3]
    	\arrow["{a^{a^fg}}"', curve={height=30pt}, dashed, from=3-1, to=1-3]
    	\arrow["{\qsi g}"description, from=3-1, to=2-2]
    \end{tikzcd}\]
\end{itemize}
A morphism between algebras are exactly a functor preserving these left Kan extensions.

It is easy to see that when $\mc C$ is itself small, the above condition is reduced to the fact that there is a left adjoint $a \dashv \eta$ to the unit $\eta : \mc C \to \ms T^{\Hcal}(\mc C)$ in $\LEX$, since for any $f : \mc A \to \mc C$, the left Kan extension $a^f$ would be induced as a composite $a^f \cong a \ms T^{\mc H}(f)$,
\[\begin{tikzcd}
	{\mc A} &&& {\mc C} \\
	{\mathsf{T}^{\Hcal}(\mc A)} & {\mc C} \\
	& {\mathsf{T}^{\Hcal}(\mc C)}
	\arrow["f", from=1-1, to=1-4]
	\arrow["\eta"', from=1-1, to=2-1]
	\arrow["f"{description}, from=1-1, to=2-2]
	\arrow["{a^f}"{description}, dashed, from=2-1, to=1-4]
	\arrow["{\ms T^{\Hcal}(f)}"', from=2-1, to=3-2]
	\arrow["{\text{id}}"{description}, from=2-2, to=1-4]
	\arrow["\eta"', from=2-2, to=3-2]
	\arrow["a"{description}, dashed, from=3-2, to=1-4]
\end{tikzcd}\]

\begin{defn}[$2$-categories of algebras for $\mathsf{T}^{\Hcal}$]
We use $\mathsf{alg}(\mathsf{T}^{\Hcal})$ to denote the 2-category of small (!) algebras of $\ms T^{\Hcal}$. $\mathsf{Alg}(\mathsf{T}^{\Hcal})$ will be the 2-category of (possibly large) algebras.
\end{defn}

As commented before, in practice $\sigma : \ms T^{\mc H} \nt \psh$ will be a morphism between relative lax-idempotent pseudomonad. Though we lack a general theory of this in the current literature, we can show the following result directly:

\begin{prop}\label{lessinterestingly}
    Every infinitary pretopos is a $\mathsf{T}^{\Hcal}$-algebra for every ${\Hcal}$, i.e. there is a forgetful functor over $\LEX$
    \[ \ms{Pretopoi}_\infty \to \ms{Alg}(\ms T^{\mc H}). \]
\end{prop}
\begin{proof}
    Let $\mc X$ be an infinitary pretopos. For any small lex category $\mc C$ with $f : \mc C \to \mc X$ in $\LEX$, consider the following diagram,
    \[\begin{tikzcd}
    	{\mc C} & {\mc X} \\
    	{\ms T^{\Hcal}(\mc C)} & {\psh(\mc C)}
    	\arrow["f", from=1-1, to=1-2]
    	\arrow["\eta"', from=1-1, to=2-1]
    	\arrow["{a^f}"{description}, dashed, from=2-1, to=1-2]
    	\arrow["\sigma"', from=2-1, to=2-2]
    	\arrow["{\lan_{\yo}f}"', from=2-2, to=1-2]
    \end{tikzcd}\]
    where we define $a^f$ to be the composite as shown above. The left Kan extension $\lan_{\yo}f$ exists in $\LEX$ precisely because infinitary pretopoi are precisely the $\psh$-algebras \cite[Prop 2.5]{lex}. By \Cref{prop:morphismofpsalg}, we have
    \[ a^f := \lan_{\yo}f \circ \sigma \cong \lan_{\yo}f \circ \lan_\eta\yo \cong \lan_{\eta} f. \]
    The unity and associativity condition of $a^f$ can be easily verified, since the above left Kan extensions are computed pointwise in $\CAT$.
\end{proof}


\begin{rem}
    The results above acknowledge a phenomenon we are very used to, and that usually even brings to some confusion in inexperienced readers of topos theory. Indeed, \textit{every} topos is lex, every topos is regular, every topos is coherent, etc. This is explained by the forgetful functor in \Cref{lessinterestingly}. Yet, not every topos is the classifying topos of a lex (regular, coherent, etc.) theory.
\end{rem}

Besides the observation that every topos is a $\ms T^{\mc H}$-algebra, the more interesting observation is the construction of classifying topos of a $\ms T^{\mc H}$-algebra:

\begin{constr}[The site associated to a small algebra]
    Every small $\mathsf{T}^\Hcal$ algebra $\mc C$ admits a canonical site structure $(\mc C, J_{\mc C}^{\Hcal})$, this will provide us with a 2-functor \[J^{\Hcal}: \mathsf{alg}(\mathsf{T}^{\Hcal}) \to \mathsf{sites}.\] The adjunction $a \dashv \eta$ has a unit $1 \to \eta a$, which induces a natural transformation $\delta : \sigma \nt \yo a$,
    \[\begin{tikzcd}
    	{\mathsf{T}^{\Hcal}(\mc C)} & {\mathsf{Psh}(\mc C)}
    	\arrow[""{name=0, anchor=center, inner sep=0}, "\yo a"', shift right=3, from=1-1, to=1-2]
    	\arrow[""{name=1, anchor=center, inner sep=0}, "\sigma", shift left=3, from=1-1, to=1-2]
    	\arrow["\delta", shorten <=2pt, shorten >=2pt, Rightarrow, from=1, to=0]
    \end{tikzcd}\]
    We define the topology $J^{\mc H}_{\mc C}$ as the \emph{least topology} that localises all the components $\delta_X$ for $X\in\ms T^{\mc H}(\mc C)$. Given a morphism $f : \mc C \to \mc D$ between small $\ms T^{\mc H}$ algebras, by definition the following diagram commutes,
    \[
    \begin{tikzcd}
        \mc C \ar[r, "f"] & \mc D \\ 
        \ms T^{\mc H}(\mc C) \ar[u, "a"] \ar[r, "\ms T^{\Hcal}f"'] & \ms T^{\mc H}(\mc D) \ar[u, "a"']
    \end{tikzcd}
    \]
    This implies that for any $X\in\ms T^{\Hcal}\mc C$, $f_!\delta_X \cong \delta_{fX}$, thus $f : (\mc C,J^{\mc H}_{\mc C}) \to (\mc D,J^{\mc H}_{\mc D})$ preserves covering, hence is a morphism of site.
\end{constr}


\begin{constr}[The site associated to a possibly large algebra]
The construction above can be extended to possibly large algebras with some additional care, which produces a 2-functor 
\[ J^{\Hcal} : \ms{Alg}(\ms T^{\mc H}) \to \ms{SITES}. \]
Let $\mc C$ be a (possibly large) $\ms T^{\mc H}$ algebra. For any $f : \mc A \to \mc C$ in $\LEX$ with $\mc A$ small, we can still write the diagram below
\[\begin{tikzcd}
	{\mc A} & {\mc C} \\
	{\mathsf{T}^{\Hcal}(\mc A)} & {\psh(\mc A)} & {\psh(\mc C)}
	\arrow["f", from=1-1, to=1-2]
	\arrow["{\eta}"', from=1-1, to=2-1]
	\arrow["\yo", from=1-2, to=2-3]
	\arrow["{a^f}"{description}, from=2-1, to=1-2]
	\arrow["\sigma"', from=2-1, to=2-2]
	\arrow["{\delta^f}"', shorten <=2pt, shorten >=2pt, Rightarrow, from=2-2, to=1-2]
	\arrow["{f_!}"', from=2-2, to=2-3]
\end{tikzcd}\]
where here $\psh(\mc C)$ denote the category of small presheaves on $\mc C$, i.e. those presheaves on $\mc C$ that can be written as a small colimit of representables, and $f_!$ is given by left Kan extension. The natural transformation $\delta^f$ is given as follows: For any $X\in\ms T^{\mc H}(\mc A)$ with $\sigma X \cong \ct_{i}\yo_{X_i}$ for $X_i\in\mc A$, by construction of $f_!$ we have
\[ f_!\sigma X \cong \ct_{i}\yo_{fX_i}. \]
Since we have a map $X_i \cong a^f\eta X_i \to a^fX$ for any $i$, this induces a morphism
\[ \delta^f_X : f_!\sigma X \cong \ct_{i}\yo_{fX_i} \to \yo_{a^fX}. \]

This brings back to a diagram that is very similar to the one of the previous construction. Now call $J^f_{\mc C}$ the least topology that inverts $\delta^f$.  The topology $J^{\Hcal}_{\mc C}$ that we would like to define is clearly, \[J_{\mc C}^{\Hcal} = \bigvee_{f: \mc A \to \mc C} J^f_{\mc C} \]
where the supremum is computed in the poset of Grothendieck topologies. This construction is always meaningful because the (large) poset of Grothendieck topologies has all supremums. The fact that a morphism of algebras $f : \mc C \to \mc D$ is covering preserving is completely similar as before, thus we get a 2-functor $J^{\Hcal}$. Indeed this extends the previous construction, since when $\mc C$ is a small lex category, this supremum coincides with $J^{\text{id}}_{\mc C}$ because as mentioned before, every $a^f$ in that case factors through $a = a^{\text{id}}$.
\end{constr}

\begin{defn}
   A (possibly large) (lex) site is \textit{Morita-small} if it is Morita equivalent to a small (lex) site. This gives us $\mathsf{SITES}_{\mathsf{M}}$, the full sub $2$-category of Morita-small sites.   An algebra $\mc C$ for $\mathsf{T}^{\Hcal}$ is \textit{Morita-small} if the resulting site is Morita-small. Morita-small algebras are collected in the $2$-category $\mathsf{Alg}(\mathsf{T}^{\Hcal})_{\mathsf{M}}$.
\end{defn}


\begin{constr}[The classifying topos of a Morita-small algebra] \label{constr:classtopos}
    The construction above associates to each Morita-small algebra $\mc C$ a Morita-small site. Then, we can compute the topos $\mathsf{Sh}(J^{\Hcal}(\mc C))$, yielding the $2$-functor $\mathsf{Cl}$ as below,
    \[\begin{tikzcd}
    	{\mathsf{Alg}(\mathsf{T}^{\Hcal})_{\mathsf{M}}} && {\mathsf{Topoi}^{\text{op}}} \\
    	& {\mathsf{SITES}_{\mathsf{M}}}
    	\arrow["{\ms{Cl}}", from=1-1, to=1-3]
    	\arrow["{J^{\mathcal{H}}}"', from=1-1, to=2-2]
    	\arrow["{\mathsf{Sh}}"', from=2-2, to=1-3]
    \end{tikzcd}\]
\end{constr}

\section{Diaconescu: Soundness and Completeness} \label{sec:sound}

Now that we have a notion of logic (which could be the semantic prescription $\Hcal$ or its more syntactic incarnation $\mathsf{T}^{\Hcal}$), and we have established some preliminary results and constructions. This way, we can initiate an investigation of the properties of logic. In this section we shall discuss a Diaconescu-type theorem. A deep analysis of the involved ingredients delivers \textit{soundness} and \textit{conceptual completeness} theorems for our fragments, and prepare the ground for a broader analysis in the style of \textit{abstract logic} (see \Cref{abstractnonsense}).

\subsection{Conceptual soundness}

From now let us focus on small $\ms T^{\mc H}$-algebras\footnote{This restriction is to just simplify the treatment of size issues. We believe everything we are going to say in this paper for small algebras will hold similarly for Morita-small algebras. However, a proper treatment of size would also digress from the conceptual message we would like to convey, thus we are content with this restriction. On a practical level, this restriction is not detrimental to the treatment of the majority of examples we have considered in this paper.}.
The goal of this subsection is to show that, the classifying topos construction provided earlier actually factors through $\WRInj(\mc H)$, 
\[ \ms{Cl} : \ms{alg}(\ms T^{\mc H}) \to \WRInj(\mc H)\op. \]
The above result states that for any $\ms T^{\mc H}$-algebra, its classifying topos indeed satisfies the semantic prescription specified by $\mc H$, and any morphism between algebras induces a geometric morphism that respects it. This is thus a \emph{conceptual soundness} result, indicating that the resulting actual syntax of logic we obtain from a semantic prescription actually conforms to the semantic prescription we start from.

\begin{lem}[Classifying topoi are injective]\label{lem:classtopoiareinj}
    The classifying topos $\ms{Cl}(\mc C)$ of a small $\ms T^{\mc H}$-algebra lies in $\WRInj(\mc H)$.
\end{lem}
\begin{proof}
    Let $\mc C$ be a small $\ms T^{\Hcal}$-algebra. By construction $\ms{Cl}(\mc C) \simeq \sh(\mc C,J^{\mc H}_{\mc C})$ is the inverter in $\Topoi$ of the following family
    \[\begin{tikzcd}
    	{\psh(\mc C)} & {\Set[\mbb O]}
    	\arrow[""{name=0, anchor=center, inner sep=0}, "{X}", shift left=3, from=1-1, to=1-2]
    	\arrow[""{name=1, anchor=center, inner sep=0}, "{aX}"', shift right=3, from=1-1, to=1-2]
    	\arrow["{\delta_X}", shorten <=2pt, shorten >=2pt, Rightarrow, from=0, to=1]
    \end{tikzcd}\]
    for all $X\in \mathsf{T}^{\Hcal}(\mc C)$. Now by construction, both $X$ and $aX$ belongs to $\WRInj(\mc H)$: The former by definition of $\ms T^{\mc H}$ and the later by the fact that $aX$ is representable. Thus by $\WRInj(\mc H)$ being closed under pseudolimits, the classifying topos $\mathsf{Cl}(\mc C)$ belongs to the logic.
\end{proof}

\begin{lem}[Classifying topoi admits presentation]\label{lem:classtopoihavepre}
    For any small $\ms T^{\mc H}$-algebra $\mc C$, the canonical inclusion $j : \ms{Cl}(\mc C) \hook \psh(\mc C)$ is an $\mc H$-presentation of $\ms{Cl}(\mc C)$.
\end{lem}
\begin{proof}
    It suffices to show that for any $f : \mc  E\to \mc F$ in $\mc H$ and $x : \mc E \to \ms{Cl}(\mc C)$, the right Kan extension $\ran_fjx$ factors through $j$. Consider the following diagram for $X\in\ms T^{\mc H}(\mc C)$,
    \[\begin{tikzcd}
    	{\mc E} & {\ms{Cl}(\mc C)} \\
    	{\mc F} & {\psh(\mc C)} & {\Set[\mbb O]}
    	\arrow["x", from=1-1, to=1-2]
    	\arrow["f"', from=1-1, to=2-1]
    	\arrow["j", from=1-2, to=2-2]
    	\arrow[dashed, from=2-1, to=1-2]
    	\arrow["{\ran_{f}jx}"', from=2-1, to=2-2]
    	\arrow[""{name=0, anchor=center, inner sep=0}, "{X}", shift left=2, from=2-2, to=2-3]
    	\arrow[""{name=1, anchor=center, inner sep=0}, "{aX}"', shift right=2, from=2-2, to=2-3]
    	\arrow["{\delta_X}", shorten <=1pt, shorten >=1pt, Rightarrow, from=0, to=1]
    \end{tikzcd}\]
    Since $j : \ms{Cl}(\mc C) \hook \psh(\mc C)$ by construction is the inverter for all $\delta_X$, it suffices to show $\ran_fjx$ inverts all $\delta_X$. We then have
    \[ X \circ \ran_fjx \cong \ran_fXjx \cong \ran_f (aX)jx \cong (aX) \circ \ran_fjx. \]
    The above uses the fact that both $X$ and $aX$ preserves the right Kan extension along $f$, and $j$ inverts $\delta_X$.
\end{proof}

\begin{thm}[Conceptual soundness]\label{thm:conceptualsound}
    The classifying topos construction $\ms{Cl}$ factors through $\WRInj_{\pw}(\mc H)$,
    \[ \ms{Cl} : \ms{alg}(\ms T^{\Hcal}) \to \WRInj_{\pw}(\mc H)\op. \]
\end{thm}
\begin{proof}
    By Lemma~\ref{lem:classtopoiareinj}, it suffices to show for any algebraic morphism $g : \mc C \to \mc D$, the induced geometric morphism lies in $\WRInj_{\pw}(\mc H)$.
    \[ \ms{Cl}(g) : \ms{Cl}(\mc D) \to \ms{Cl}(\mc C). \]
    Consider an $f : \mc E \to \mc F$ in $\mc H$ and $x : \mc E \to \ms{Cl}(\mc D)$. We have the following diagram,
    \[\begin{tikzcd}
    	{\mc E} & {\ms{Cl}(\mc D)} & {\psh(\mc D)} \\
    	{\mc F} & {\ms{Cl}(\mc C)} & {\psh(\mc C)}
    	\arrow["x", from=1-1, to=1-2]
    	\arrow["f"', from=1-1, to=2-1]
    	\arrow["j", hook, from=1-2, to=1-3]
    	\arrow["{\ms{Cl}(g)}"{description}, from=1-2, to=2-2]
    	\arrow["{\psh(g)}", from=1-3, to=2-3]
    	\arrow["\ran_fx"description, dashed, from=2-1, to=1-2]
    	\arrow["\ran_{f}jx"{description, near end}, curve={height=6pt}, dashed, from=2-1, to=1-3]
    	\arrow[from=2-1, to=2-2]
    	\arrow["j", hook, from=2-2, to=2-3]
    \end{tikzcd}\]
    The right square commutes by the fact that $g$ is an algebraic morphism. To show $\ms{Cl}(g)$ preserves the right Kan extension along $f$, we observe that
    \begin{align*}
      j \ms{Cl}(g) \ran_fx   \cong & \psh(g) j \ran_fx \\
         \cong & \ran_{f}\psh(g)jx \\ 
         \cong & \ran_fj\ms{Cl}(g)x \\ 
         \cong & j\ran_f\ms{Cl}(g)x.
    \end{align*}
    The above uses the fact that both $j$ and $\psh(g)$ preserves right Kan extensions along $f$: The former holds by Lemma~\ref{lem:classtopoihavepre}, and the later holds since $\psh(f)$ is ethereal. Since $j$ is an embedding, we indeed have $\ms{Cl}(g) \ran_fx \cong \ran_f\ms{Cl}(g)x$, hence $\ms{Cl}(g)$ belongs to $\WRInj_{\pw}(\mc H)$.
\end{proof}



\subsection{Parametric Diaconescu's theorem} \label{diaconescu}

Recall from \Cref{lessinterestingly} there is a forgetful functor from $\Topoi$ to $\ms T^{\mc H}$-algebras. Diaconescu's theorem states that this is relative adjoint to the classifying topos construction:

\begin{thm}[Diaconescu]
    We have a relative pseudoadjunction as below,
    \[\begin{tikzcd}
    	{\mathsf{alg}(\mathsf{T}^{\Hcal})} & {\mathsf{Topoi}\op} \\
    	{\mathsf{Alg}(\mathsf{T}^{\Hcal})}
    	\arrow[""{name=0, anchor=center, inner sep=0}, "{\ms{Cl}}", from=1-1, to=1-2]
    	\arrow[from=1-1, to=2-1]
    	\arrow[""{name=1, anchor=center, inner sep=0}, from=1-2, to=2-1]
    	\arrow["\dashv"{anchor=center, rotate=-90}, draw=none, from=1, to=0]
    \end{tikzcd}\]
\end{thm}
\begin{proof}
    By the above construction, for any topos $\mc E$ and small algebra $\mc C$,
    \[ \Topoi(\mc E,\mathsf{Cl}(\mc C)) \simeq \ms{SITES}(J^{\mc H}(\mc C),\mc E), \]
    where a morphism of site $F : J^{\mc H}(\mc C) \to \mc E$ is a left exact functor that preserves covers. Equivalently, this means that when we take the following left Kan extension
    \[
    \begin{tikzcd}
        \mc C \ar[d, "\eta"'] \ar[r, "F"] & \mc E \\ 
        \mathsf{T}^{\Hcal}(\mc C) \ar[ur, dashed, "a^F"']
    \end{tikzcd}
    \]
    $a^F$ inverts maps of the form $\delta_X : X \to \eta aX$ for $X\in\ms T^{\mc H}(\mc C)$. In other words, $a^F \cong a^F\eta a$, thus $F$ is a morphism between $\ms T^{\mc H}$-algebras. This way,
    \[ \Topoi(\mc E,\mathsf{Cl}(\mc C)) \simeq \ms{Alg}(\ms T^{\mc H})(\mc C,\mc E). \qedhere \]
\end{proof}

\begin{rem}
    When $\mc H$ is taken to be $\mc H_{\mr{eth}}$, the above relative pseudoadjunction specialises to the following diagram on the left, which corresponds to the plain version of Diaconescu's theorem, stating that geometric morphisms into a free topos corresponds exactly to left exact functors out of the site.
    \[\begin{tikzcd}
    	\lex & {\mathsf{Topoi}\op} & {\Topoi\op} & {\Topoi\op} \\
    	\LEX && {\ms{Pretopoi}_\infty}
    	\arrow[""{name=0, anchor=center, inner sep=0}, "{{\ms{Cl}}}", from=1-1, to=1-2]
    	\arrow[from=1-1, to=2-1]
    	\arrow[""{name=1, anchor=center, inner sep=0}, from=1-2, to=2-1]
    	\arrow[""{name=2, anchor=center, inner sep=0}, "{\ms{Cl}}", from=1-3, to=1-4]
    	\arrow[from=1-3, to=2-3]
    	\arrow[""{name=3, anchor=center, inner sep=0}, from=2-3, to=1-4]
    	\arrow["\dashv"{anchor=center, rotate=-90}, draw=none, from=1, to=0]
    	\arrow["\dashv"{anchor=center, rotate=-90}, draw=none, from=3, to=2]
    \end{tikzcd}\]
    Though we haven't provided a full treatment extending to the Morita-small algebras, it should be clear that when taking the class $\mc H_{\mr{all}}$, by replacing small algebras to Morita-small algebras, the above relative pseudoadjunction will specialise to the diagram on the right above. Here $\ms{Cl}$ is simply identity. This explains the so-called \emph{logoi-topoi duality}; see for instance~\cite{anel2021topo}.
\end{rem}

\subsection{Conceptual completeness}
\label{conceptual}

In this section we shall relate the construction of the classifying topos to Makkai's \textit{conceptual completeness} \cite{makkai88} for first order logic, and make it parametric in the choice of the logic. 

Let us take the example of the coherent fragment. Let $\mc X$ be a coherent topos. We will first investigate what is the result for the construction of syntactic category by viewing it as an object in $\WRInj(\mc H_{\mr{flat}})$. We can indeed show that these are exactly the coherent objects in $\mc X$:

\begin{prop}\label{prop:synofcoh}
    For any coherent topos $\mc X$, $\ms{Syn}^{\mc H_{\mr{flat}}}(\mc X)$ is equivalent to the full subcategory of coherent objects in $\mc X$.
\end{prop}
\begin{proof}
    Recall by construction, an object $X$ in $\ms{Syn}^{\mc H_{\mr{flat}}}(\mc X)$ is one that preserves the right Kan extensions along flat geometric morphisms, which means that for any flat map $f : \mc E \to \mc F$ and $x : \mc E \to \mc X$, we have the following situation,
    \[\begin{tikzcd}
    	{\mc E} & {\mc X} \\
    	{\mc F} & {\Set[\mbb O]}
    	\arrow["x", from=1-1, to=1-2]
    	\arrow["f"', from=1-1, to=2-1]
    	\arrow["X", from=1-2, to=2-2]
    	\arrow["{\ran_fx}"{description}, dashed, from=2-1, to=1-2]
    	\arrow["{f_*x^*X}"', dashed, from=2-1, to=2-2]
    \end{tikzcd}\]
    where the lower triangle commutes,
    \[ X \circ \ran_fx \cong f_*x^*X. \]

    Now suppose $X$ is a coherent object in $\mc X$. Let $\mc X \simeq \sh(\mc C,J)$ where $\mc C$ is the pretopos of coherent objects in $\mc X$ and $J$ the coherent topology. From Theorem~\ref{thm:cohformal}, we know $\ran_fx$ is computed pointwise, and in fact induced by the right Kan extension on $\psh(\mc C)$. This way,
    \[ (\ran_fx)^*X \cong (\lan_{\yo}f_*x^*\yo)X \cong f_*x^*X, \]
    which implies $X\in\ms{Syn}^{\mc H_{\mr{flat}}}(\mc X)$.

    On the other hand, to show for any $X\in\ms{Syn}^{\mc H_{\mr{flat}}}(\mc X)$, $X$ will be coherent, we use the connection between maps preserving the right Kan extension along $\mc H_{\mr{flat}}$ and ultrafunctors discussed previously. Consider the flat embedding $\Set^I \hook \sh(\beta I)$, with the following diagram
    \[\begin{tikzcd}
    	{\Set^I} & {\mc X} \\
    	{\sh(\beta I)} & {\Set[\mbb O]}
    	\arrow["x", from=1-1, to=1-2]
    	\arrow["f"', from=1-1, to=2-1]
    	\arrow["X", from=1-2, to=2-2]
    	\arrow["{\ran_fx}"{description}, dashed, from=2-1, to=1-2]
    	\arrow["{f_*x^*X}"', dashed, from=2-1, to=2-2]
    \end{tikzcd}\]
    Since $\sh(\beta I)$ has enough points, $X$ makes the lower triangle commutes iff for all $\mu\in\beta I$ we have 
    \[ X \circ \ran_fx \circ \mu \cong (f_*x^*X) \circ \mu. \]
    By viewing the above two geometric morphisms $\Set \to \Set[\mbb O]$ as sets, and by the discussion at the end of Section~\ref{sec:logic} relating $\ran_f(-)$ with ultraproducts, this exactly means that
    \[ X(\int xd\mu) \cong \int Xx d\mu, \]
    In other words, the evaluation functor $\ms{pt}(\mc X) \to \Set$ induced by $X$ preserves ultraproducts. Such an object $X$ must be coherent is the consequence of the duality result in~\cite{makkai1987stone}.
\end{proof}

\begin{thm}\label{thm:cohalgebra}
    $\mathsf{Pretopoi}$ is the 2-category of algebras for $\mathsf{T}^{\mathcal{H}_{\text{flat}}}$.
\end{thm}
\begin{proof}
    For any small lex category $\mc C$, $\ms T^{\mc H}(\mc C)$ by Proposition~\ref{prop:synofcoh} is the category of coherent objects in $\psh(\mc C)$. We show this is the free pretopos generated by $\mc C$. Notice by the results in~\cite[Sec. 5]{lex}, being a pretopos is an exactness property over $\LEX$. By Thm. 7.7 of the \emph{loc. cit.}, it suffices to show $\ms T^{\mc H}(\mc C)$ is the closure of $\mc C$ in $\psh(\mc C)$ under finite limits, finite coproducts, and effective quotients. Let us denote the latter category as $P\mc C$. Consider the inclusion of site,
    \[ \mc C \hook (P\mc C,J), \]
    where $J$ is the canonical topology on $\psh(\mc C)$ restricted to $P\mc C$, which coincides with its coherent topology. This way, the topology restricted to $\mc C$ is the trivial topology. Furthermore, by construction every object in $P\mc C$ is covered by a coproduct of objects in $\mc C$. Hence, this is a dense morphism of site, and by the comparison lemma they induces an equivalence of topoi,
    \[ \psh(\mc C) \simeq \sh(P\mc C,J). \]
    Now by~\cite[D3.3.7]{elephant2}, the coherent objects in $\sh(P\mc C,J)$ are exactly the images of $P\mc C$ under Yoneda. This implies that $\ms T^{\mc H}(\mc C) \simeq P\mc C$, hence $\ms T^{\mc H}(\mc C)$ is the free pretopos generated by $\mc C$.
\end{proof}

\begin{rem}[There are more fragments than doctrines]\label{rem:morelogicthandoctrine}
    Following the proof of the above, it follows that both $\mc H_{\mr{flat}}$ and $\mc H_{\beta}$ induce the \emph{same} lax-idempotent pseudomonad, viz. the free pretopos construction over $\lex$. However, from \Cref{prop:betanotflat} we know that $\mc H_{\mr{flat}}$ and $\mc H_{\beta}$ do not generate the same Kan injectivity class in $\Topoi$. 
\end{rem}


\begin{cor}\label{conceptualcompleteness}
    For any coherent topoi $\mc X,\mc Y$, a map $f : \mc X \to \mc Y$ belongs to $\WRInj(\mc H_{\mr{flat}})$ iff $f^*$ preserves coherent objects, i.e. there is a full embedding
    \[ \ms{Pretopoi}\op \hook \WRInj(\mc H_{\mr{flat}}). \]
\end{cor}
\begin{proof}
    Similarly, this is a consequence of the duality result in~\cite{makkai1987stone}.
\end{proof}

\begin{rem}[The relation to conceptual completeness] \label{relationtoconceptual}
Let us discuss more precisely in what sense the corollary above can be understood as a form of conceptual completeness. The usual understanding of conceptual completeness is given by the fact that Pretopoi are $2$-fully faithfully embedded in ultracategories,
\[ \ms{Pretopoi}\op \hook \ms{Ult}. \]
The general idea of the theorem is the fact that from the \textit{semantics trace} of the category of models one can recover the whole theory that defines it. 

Now, one could claim that on the right of the fully faithful embedding we have shown above, 
\[ \ms{Pretopoi}\op \hook \WRInj(\mc H_{\mr{flat}}), \]
we are not discussing ultracategories, and that we are -- at most -- discussing some class of topoi, hence it is not that clear how to connect the corollary above to Makkai's theorem. Yet, when reading the proofs above in detail, one realises that this is not exactly the case. The $\beta$-completeness of each of those topoi is specifying the ultrastructure on its category of points, and the restriction to morphism of injectives is selecting ultrafunctors. So, while indeed on the right we are specifying \textit{theories} (as opposed to models) our notion of logic is inherently making claims about their semantics behaviour (see \cite[Sec. 2]{diliberti2022geometry}). 

We thus obtain an \textit{analysis} of the usual conceptual completeness via a factorization of its embedding,
\[ \ms{Pretopoi}\op \hook \WRInj(\mc H_{\mr{flat}}) \hook \ms{Ult}. \]
We shall claim that this is a more \textit{adequate} notion of conceptual completeness, as it bounds the axioms (reified by the syntactic category) to their \textit{prescribed} semantic behaviour, while the usual conceptual completeness is making a further step into semantics, for the lack of a finer description of logical systems. Our point here is that a conceptual completeness (as opposed to a \textit{reconstruction theorem} or a \textit{duality theorem}) should be understood as a form of Beth-definability for (prescribed) semantic behaviours.
\end{rem}

\begin{rem}[An unsatisfactory result] \label{6.3.6}
It goes without saying that it's quite unsatisfying for us that our proof of \Cref{conceptualcompleteness} relies on the classical conceptual completeness, and that we haven't managed to deliver a proof that completely relies on the introduced framework. This is also the reason why large portion of our table in the introduction is are so hard to fill and are left to conjectures. Yet, we believe that further work could fill this gap, and deliver a more modular approach to completeness-like theorems.
\end{rem}

Finally, we can end this section by formulating a definition and asking a question. Despite currently not delivering any self contained answer, notice how such notions could not even be defined without the technology of the paper (and indeed have only been investigated case-by-case).

\begin{defn}
    A bounded logic $\mathcal{H}$ enjoys conceptual completeness if the 2-functor exhibiting conceptual soundness $\mathsf{alg}{(\mathsf{T}^{\mathcal{H}})}\op \to \WRInj(\mc H)$ is in fact 2-fully faithful.
\end{defn}

\begin{rem}
    Again as mentioned earlier, the above definition is confined to the bounded case simply due to the fact that we are not actively dealing with size issues in this paper -- not because we think they are particularly hard, but just this treatment falls out of the scope of this paper. For the general case of conceptual completeness, we should ask the (tentative) 2-functor $\ms{Alg}(\ms T^{\mc H})\op_{\ms M} \to \WRInj(\mc H)$ to be 2-fully faithful.
\end{rem}

\begin{quest} \label{question}
    What logic $\Hcal$ enjoy conceptual completeness?
\end{quest}

\begin{exa}
    As shown in \Cref{class:essalg}, the essentially algebraic fragment $\Hcal_{\text{all}}$ enjoys conceptual completeness. If we ignore size issues, the geometric fragment $\Hcal_{\text{eth}}$ also enjoys conceptual completeness as shown in \Cref{chargeom}. Of course, as indicated in \Cref{conceptualcompleteness}, the classical conceptual completeness for coherent logic implies conceptual completeness for $\mc H_{\mr{flat}}$ in our sense.
\end{exa}


\section{Catharsis and Sipario} \label{sec:future}

This short section serves as an outro of the paper. Developing the foundations of this theory has required a quite deep leap of faith: while we would have wanted to jump directly to the last two sections (and even the next paper (see \ref{abstractnonsense})), finding the correct definitions and growing a theory that could encompass all the examples has required to start with a more foundational take. The final result is a very broad project, where many different routes can be taken. There are four main directions.

\subsection{Broaden our understanding of formal category theory in $\mathsf{Topoi}$} 

Below would be very much a further investigation of the work we initiated in \Cref{sec:universal}.

\begin{enumerate}
    \item[($\mathsf{point}$)] In \Cref{pointwise} we have seen a bit \textit{en passant} that the notion of pointwise Kan extension based on the equipment coincides with the one internally captured by comma objects. In the paper, we do not explore the consequences of this fact, but it would be interesting if this has interesting repercussions to the theory, as in principle the \textit{internalization} of pointwise Kan extension should influence their behaviour.
    \item[$(\mathsf{lifts})$] It would be also interesting to study better Kan lifts in  $\mathsf{Topoi}$, as we expect them to be better behaved then in $\CAT$. One source of this expectation is given by the fact that there seems to be more naturally arising coreflective subcategories of $\Topoi$, with Boolean topoi being one such example.
\end{enumerate}

\subsection{Complete several aspects of the taxonomy initiated in \Cref{sec:logicshowoff}} There is quite an asymmetry in the level of development of each example in \Cref{sec:logicshowoff}. For some fragments we can give a complete and satisfying description of their injectives and their algebras and for others we can't. In the table below we condense as much as possible of this information.

\begin{table}[!h]\renewcommand{\arraystretch}{1.3}
\begin{tabular}{|c|c|c|}
\hline
\textbf{Logic} --  \textbf{$\mathcal{H}$}  & $\mathsf{WRInj}_{\pw}(\mathcal{H})$  & $\ms{alg}(\ms T^{\mc H})/\mathsf{Alg}(\ms T^\mathcal{H}$) \\ \hline

geometric \hfill  $\mathcal{H}_{\text{eth}}$ (\ref{defeth}) & $\mathsf{Topoi}$  \hfill (\ref{chargeom}) & \textsf{Pretopoi}$_\infty$ \hfill (\ref{examplesofsyns}) \\ \hline

integral \hfill  $\mathcal{H}_{\text{cl}}$ (\ref{exa:complementedfragment}) & $\mathsf{TocTopoi}$ \hfill  (\ref{cor:classtot}) & \textsf{ConLex} \hfill (\ref{examplesofsyns}) \\ \hline

ess. alg. \hfill  $\mathcal{H}_{\text{all}}$ (\ref{defall}) & $\mathsf{Free}^{*}$ \hfill  (\ref{class:essalg}) & \textsf{Lex} \hfill (\ref{examplesofsyns}) \\ \hline

ess. alg. w. $\bot$ \hfill  $\mathcal{H}_{\text{dom}}$ (\ref{defdom}) & $\mathsf{ClFree}^{*}$  \hfill (\ref{thm:elsealgebraic}) & \textsf{siLex} \hfill (conj!) \\ \hline

disjunctive \hfill  $\mathcal{H}_{\text{inn}}$ (\ref{definn}) & $\mathsf{DisjTopoi}^{*}$ \hfill  (conj!) & \textsf{Lexten}$_\infty$ \hfill (conj!) \\ \hline

fin. disj. \hfill  $\mathcal{H}_{\text{pure}}$ (\ref{defpure}) & $\mathsf{fDisjTopoi}^{*}$ \hfill  (conj!) &  \textsf{Lexten} \hfill (conj!) \\ \hline

regular \hfill   $\mathcal{H}_{\text{matte}}$ (\ref{defmatte}) & $\mathsf{RegTopoi}^{*}$ \hfill  (conj!) & \textsf{EffReg} \hfill (conj!) \\ \hline

coherent \hfill $\mathcal{H}_{\text{flat}}$ (\ref{defflat}) & $\mathsf{CohTopoi}^{*}$ \hfill  (conj!) & \textsf{Pretopoi} \hfill (\ref{thm:cohalgebra}) \\ \hline

\end{tabular}
\label{tab:summary}
\end{table}

\begin{enumerate}
\item[($\WRInj_{\pw}(\Hcal)$)] The difficulty in providing a complete classification of these 2-categories is displayed -- for example -- by \Cref{thm:elsealgebraic} where we show an exact correspondence between $\WRInj_{\pw}(\mc H_{\mr{dom}})$ and closed subtopoi of free topoi. In the implication $(1) \Rightarrow (2)$, the only strategy that is known to us requires to show that every $\mc X\in\WRInj_{\pw}(\mc H)$ admits an $\Hcal$-embedding into a classifying topos in this fragment of logic. Unfortunately, we haven't managed to prove this result for every fragment, and thus we must leave the table incomplete. In the table where we haven't provided a full characterisation of $\WRInj_{\pw}(\mc H)$, we conjecture the topoi in $\WRInj_{\pw}(\mc H)$ are up to coadjoint retracts the classifying topoi of the corresponding fragment of logic.
\item[($\mathsf{alg}(\mathsf{T}^{\Hcal})$)] Similarly, for several fragments we haven't provided a characterisation of algebras for the (relative) pseudomonad $\ms T^{\Hcal}$. In this case, the difficulty in delivering this result is perfectly exemplified by  \Cref{prop:synofcoh}, where we had to invoke Makkai's conceptual completeness in order to compute the syntactic category. Of course, the class of maps $\mc H_{\mr{flat}}$ we choose is much more general the $\beta$-maps, thus \emph{a priori} a simpler and more conceptual computation could be available. Such a method is most likely to provide a general solution to all the cases we have listed. The table does contain our computation/conjecture in each case. The unfamiliar objects are $\ms{ConLex}$, which consists of those lex categories that can arise as the subcategory of connected objects in a totally connected topos, and $\ms{siLex}$, which consists of lex categories with a strict initial object.
\end{enumerate}

We believe it is crucial to complete the table above. To succeed in such computations will probably require genuinely new insights and development of techniques. Such results will deepen our understanding of the relationship between the syntax and semantics of logic, for example in addressing \Cref{question} and \Cref{6.3.6}.

\subsection{Symbolic logic} \label{concretenonsense}

We have shown a fragment $\mc H$ induces a doctrine $\ms T^{\mc H}$, which can be viewed as recovering a form of categorical syntax for this logic. One might go a step further to introduce a proper \emph{deductive system} over a \emph{concrete syntax} associated to $\mc H$, so that $\ms T^{\mc H}$ can then be understood as the construction of syntactic categories in the traditional sense~\cite[D1.4]{elephant2}. For instance, this will answer our quest in \Cref{exa:complementedfragment} to provide a concrete syntax for the exotic fragment $\mc H_{\mr{cl}}$. A systematic approach will provide a strong connection between logic in the traditional sense, which is about the manipulation of symbols, and \emph{abstract logic}, which is the perspective of this paper.

\subsection{Abstract logic} \label{abstractnonsense}

The perspective of this paper (that of studying logic abstractly and in a modular way) is not entirely new to the literature. Before diving further we shall discuss the connection between our approach and existing attempts. 

After Tarski's original definition of validity and satisfaction, many called for a general treatment of logical systems (notably Mostowski \cite{mostowski1957generalization}). This led to the tradition of \textit{abstract logic}, which has several incarnations.

One is the school that studies \textit{abstract satisfaction relations} $(\mathcal{L}, \models_{\mathcal{L}})$ \cite[Chap. 2]{barwise2017model}. In this approach, both syntax and semantics are specified independently, and are bond by the satisfaction relation $\models_{\mathcal{L}}$. The most celebrated result of this tradition is Lindström theorem \cite[Chap. 3]{barwise2017model}, which provides a characterization of first order logic in terms of some properties of its abstract satisfaction relations. The more recent theory of institutions \cite{diaconescu2008institution} should be understood as a natural prosecution of this tradition. 

Another school is that of \textit{algebraic logic} \cite{blok1989algebraizable}. Algebraic logic deal with varieties of posets, which are understood as putative Lindenbaum-Tarski algebras of fragments of propositional logic, and investigates the properties of those varieties to infer properties of the fragment of logic. One relevant cluster of results of this school is the reduction of Craig interpolation-type theorems \cite{craig1957three} to the \textit{amalgamation} property \cite[2.6.1]{hoogland2001definability}.

Our approach feels closest to algebraic logic. Indeed, we are carving out relevant classes of topoi, and in a similar spirit, they are carving out relevant classes of lattices. Moreover, the jump from lattices to categories \textit{precisely} fills the gap between propositional and predicate logic. However, \Cref{rem:morelogicthandoctrine} shows that the environment we are working with has a richer semantic aspect. Besides this, the advantage of our approach to that of abstract satisfaction relations is that it captures more precisely what a logic \textit{is}, as opposed to only describing its \textit{semantic shadow} via the satisfaction relation. As mentioned multiple times, for us a logic is (a) its intended semantic behaviour $\Hcal$, (b) its collection of theories $\WRInj(\Hcal)$, (c) its syntactic description $\mathsf{T}^{\Hcal}$. Most importantly, it is the combination of (a),(b) and (c), which interact with each other under the unified umbrella we propose in the paper. 

Importing (some of) the classical results that these traditions have produced to our framework, offering on one hand a more \textit{syntactic} understanding of them, and on the other a generalization to predicate logic is currently the main focus of the two authors for the near future.

\bibliography{thebib}
\bibliographystyle{alpha}

\end{document}